\documentclass[11pt]{article} 

\usepackage[leqno]{amsmath}
\usepackage{amsthm}
\usepackage{amssymb}
\usepackage{graphicx}
\usepackage{calc} % simple arithmetics in latex
\usepackage{mathtools}
\usepackage{enumitem}
\usepackage{ifpdf} % define \ifpdf
\usepackage{ifthen}
\usepackage{braket}
\usepackage{xcolor}
\usepackage[all]{xy}
\usepackage{fancyhdr}
\usepackage{tikz}
\usepackage{todonotes} %[disable]
\usepackage{dsfont}
\usepackage{subcaption}
\usepackage{sidecap}
\ifpdf
\usepackage[pdftex]{hyperref}
\hypersetup{plainpages=false,unicode=true,pdffitwindow=true}
\else
\usepackage[dvips]{hyperref}
\fi

\usepackage{algorithm}
\usepackage[noend]{algpseudocode}

\theoremstyle{plain}
\newtheorem{theorem}{Theorem}[section]

\newtheorem{corollary}[theorem]{Corollary}
\theoremstyle{definition}
\newtheorem{definition}[theorem]{Definition}
\newtheorem{example}[theorem]{Example}

\newtheorem{remark}[theorem]{Remark}

\renewcommand{\bar}[1]{\overline{#1}}

\renewcommand{\hat}[1]{\widehat{#1}}

\usepackage[section]{placeins}
\usepackage{tabularx,ragged2e,booktabs,caption}
\usepackage{hyperref}
\usepackage{algorithm}
\usepackage{footnote}
\usepackage{bbm}
\numberwithin{equation}{section}
\numberwithin{figure}{section}
\numberwithin{table}{section}
\numberwithin{algorithm}{section}
\usepackage{authblk} % for author format
\usepackage{tikz}
\usetikzlibrary{calc}

\definecolor{helmholtzdarkblue}{RGB}{0,40,100} 
\definecolor{helmholtzlightblue}{RGB}{20,200,255} 
\definecolor{helmholtzhighlight}{RGB}{205,238,251}
\definecolor{helmholtzpaleblue}{RGB}{236,251,253} 
\definecolor{logoblue}{RGB}{0,90,160}  % Helmholtz logo blue
\definecolor{mint}{RGB}{5,229,186}  % Helmholtz mint
\colorlet{helmholtzmint}{mint}

\usepackage{adjustbox}  
\usetikzlibrary{shapes}
\usetikzlibrary{shapes.geometric, arrows}
\tikzstyle{startstop} = [rectangle, rounded corners, minimum width=3cm, minimum height=1cm, text centered, text width=4.5cm, draw=helmholtzdarkblue, fill=blue!30]
\tikzstyle{io} = [trapezium, trapezium left angle=70, trapezium right angle=110, minimum width=3cm, minimum height=1cm, text centered, draw=black, fill=blue!30]
\tikzstyle{process} = [rectangle, rounded corners, minimum width=3cm, minimum height=1cm, text centered, text width=10cm, draw=black, fill=red!10]
\tikzstyle{decision} = [rectangle, minimum width=2.5cm, minimum height=5cm, text centered, text width=3cm, draw=black, fill=blue!30]
\tikzstyle{alternative} =  [rectangle, minimum width=2.5cm, minimum height=1cm, text centered, text width=3cm, draw=black, fill=blue!30]
\tikzstyle{alternative2} = [diamond, fill=green!10, text width=2.2cm, text centered]
\tikzstyle{arrow} = [thick,->,>=stealth]
\tikzstyle{circ} = [circle, text centered, draw=black, text width=1.7cm, fill=blue!10]
\tikzstyle{box} = [rectangle, minimum width=2.5cm, minimum height=5cm, text centered, text width=10cm, draw=helmholtzdarkblue, fill=blue!20]
\tikzstyle{groupingred} = [rectangle, fill=black!10,minimum width=15cm, minimum height=17cm,black!10]
\tikzstyle{groupingblue} = [rectangle, fill=blue!10,minimum width=38cm, minimum height=10cm,blue!10]

\definecolor{helmholtzdarkblue}{RGB}{0,40,100} % primary colour `Helmholtz dark blue`
\colorlet{color1}{helmholtzdarkblue}
\newcommand{\rset}{\mathbb{R}}
\newcommand{\nset}{\mathbb{N}}
\newcommand{\zset}{\mathbb{Z}}

\newcommand{\COMMA}{,}

\newcommand{\Ordo}[1]{{\mathcal{O}}\left(#1\right)}

\makeatletter
\def\BState{\State\hskip-\ALG@thistlm}
\makeatother

\title{Automated Importance Sampling via Optimal Control for Stochastic Reaction Networks:  A Markovian Projection--based Approach}

\author[1]{Chiheb Ben Hammouda}
\author[2]{Nadhir Ben Rached}
\author[3,4]{Ra\'ul Tempone}
\author[5]{Sophia Wiechert\thanks{Corresponding author: wiechert@uq.rwth-aachen.de}}
\affil[1]{Utrecht University, Mathematical Institute, 3584 CD Utrecht, The Netherlands.}
\affil[2]{University of Leeds, School of Mathematics, Woodhouse, Leeds LS2 9JT, UK.}
\affil[3]{King Abdullah University of Science and Technology (KAUST), Computer, Electrical and Mathematical Sciences \& Engineering Division (CEMSE), Thuwal 23955-6900, Saudi Arabia.}
\affil[4]{RWTH Aachen University, Alexander von Humboldt Professor in Mathematics for Uncertainty Quantification, 52062 Aachen, Germany.} 
\affil[5]{RWTH Aachen University, Chair of Mathematics for Uncertainty Quantification, Pontdriesch 14-16,
	52062 Aachen, Germany.}

\begin{document} 
    \date{}
\maketitle

\begin{abstract} 
   We propose a novel alternative approach to our previous work (Ben Hammouda et al., 2023) to improve the efficiency of Monte Carlo (MC) estimators for rare event probabilities for stochastic reaction networks (SRNs). In the same spirit of (Ben Hammouda et al., 2023),  an efficient path-dependent measure change is derived based on a  connection between determining optimal importance sampling (IS) parameters within a class of probability measures and a stochastic optimal control formulation, corresponding to solving a variance minimization problem. 
   In this work, we propose a novel approach to address the encountered curse of dimensionality by mapping the problem to a significantly lower-dimensional space via a Markovian projection (MP) idea. The output of this model reduction technique is a low-dimensional SRN (potentially even one dimensional) that preserves the marginal distribution of the original high-dimensional SRN system. The dynamics of the projected process are obtained by solving a related optimization problem via a discrete $L^2$ regression. By solving the resulting projected Hamilton--Jacobi--Bellman (HJB) equations for the reduced-dimensional SRN, we obtain projected IS parameters, which are then mapped back to the original full-dimensional  SRN system, resulting in an efficient IS-MC estimator for rare events probabilities of the full-dimensional SRN.  Our analysis and numerical experiments reveal that the proposed MP-HJB-IS approach substantially reduces the MC estimator variance, resulting in a lower computational complexity in the rare event regime than standard MC estimators.
   
\end{abstract}

\thispagestyle{plain}

\setcounter{tocdepth}{1}

\section{Introduction}

This paper proposes an efficient estimator for rare event probabilities for a particular class of continuous-time Markov processes, stochastic reaction networks (SRNs). We design an automated importance sampling (IS) approach based on the approximate explicit tau-leap (TL) scheme to build a Monte Carlo (MC) estimator for rare event probabilities of SRNs. The used IS change of measure was introduced in \cite{ben2023learning}, wherein the optimal IS controls were determined via a stochastic optimal control (SOC) formulation. In that same work, we also presented a learning-based approach to avoid the curse of dimensionality. Building on that work, we propose an alternative method for high-dimensional SRNs that leverages dimension reduction through Markovian projection (MP) and then recovers the optimal IS controls of the full-dimensional SRNs as a mapping from the solution in lower-dimensional space, potentially one. To the best of our knowledge, we are the first to establish the MP framework for the SRN setting to solve an IS problem.

An SRN (refer to Section~\ref{sec:SRNs} for a brief introduction and \cite{ben2020hierarchical} for more details)
describes the time evolution of a set of species through reactions and can be found in a wide range of applications, such as biochemical reactions, epidemic processes \cite{brauer2001mathematical,anderson2015stochastic}, and transcription and translation in genomics and virus kinetics \cite{srivastava2002stochastic,hensel2009stochastic}. For  a $d$-dimensional SRN, $\mathbf{X}:[0,T]\rightarrow \nset^d$, with the given final time $T>0$, we aim to determine accurate and computationally efficient MC estimations for the expected value $\mathbb{E}[g(\mathbf{X}(T))]$. The observable $g:\nset^d\to\rset$ is a given scalar function of $\mathbf{X}$, where indicator functions $g(\mathbf{x})=\mathbbm{1}_{\{\mathbf{x} \in \mathcal{B}\}}$ are of interest to estimate the rare event probability $\mathbb{P}(\mathbf{X}(T)\in\mathcal{B})\ll 1$, where $\mathcal{B} \subset \rset^d$.

	The quantity of interest, $\mathbb{E}[g(\mathbf{X}(T))]$, is the solution to the corresponding Kolmogorov backward equations \cite{bayer2016efficient}. Solving these ordinary differential equations (ODEs) in closed form is infeasible for most SRNs; thus, numerical approximations based on discretized schemes are used to derive solutions.
 A drawback of these approaches is that, without using dimension reduction techniques, the computational cost scales exponentially with the number of species $d$. To avoid the curse of dimensionality, we propose estimating $\mathbb{E}[g(\mathbf{X}(T))]$ using MC methods.

Numerous schemes have been developed to simulate the exact sample paths of SRNs. These include the stochastic simulation algorithm introduced by Gillespie in \cite{gillespie1976general} and the modified next reaction method proposed by Anderson in \cite{anderson2007modified}. However, when SRNs involve reaction channels with high reaction rates, simulating exact realizations of the system can be computationally expensive. To address this issue, Gillespie \cite{gillespie2001approximate} and Aparicio and Solari  \cite{aparicio2001population} independently proposed the explicit-TL method (see Section~\ref{sec:exp_tau}), which approximates the paths of $\mathbf{X}$ by evolving the process with fixed time steps while maintaining constant reaction rates within each time step.
Additionally, other simulation schemes have been proposed to handle situations with well-separated fast and slow time scales \cite{cao2005trapezoidal,rathinam2007reversible,abdulle2010chebyshev,ahn2013implicit,moraes2016multilevel,hammouda2017multilevel}.

In order to compute MC estimates of $\mathbb{E}[g(\mathbf{X}(T))]$ more efficiently, different  variance reduction techniques have been proposed in the context of SRNs. In the spirit of the multilevel MC (MLMC) idea \cite{giles2008multilevel,giles2015multilevel}, various  MLMC-based methods \cite{Anderson2012,lester2015adaptive,moraes2016multilevel,hammouda2017multilevel,ben2020importance} have been introduced to overcome different challenges in this context.  Moreover, as the naive MC and MLMC estimators have high computational costs when used for estimating rare event probabilities, different IS approaches \cite{kuwahara2008efficient,gillespie2009refining,roh2010state,daigle2011automated,cao2013adaptively,gillespie2019guided,roh2019data,ben2023learning} have been proposed. 

To estimate various statistical quantities efficiently for SRNs (specifically rare event probabilities), we use the path-dependent IS approach originally introduced in \cite{ben2023learning}. This class of probability measure change is based on modifying the rates of the Poisson random variables used to construct the TL paths. In \cite{ben2023learning}, it is shown how optimal IS controls are obtained by minimizing the second moment of the IS estimator (equivalently, the variance), representing the cost function of the associated SOC problem, and that the corresponding value function solves  a dynamic programming relation (see Section~\ref{sec:Approach Formulation} for revising these results). In this work, we generalize the discrete-time dynamic programming relation by a set of continuous-time ODEs, the Hamilton--Jacobi--Bellman (HJB) equations, allowing the formulation of optimal IS controls in continuous time. Compared to the discrete-time IS control formulation presented in \cite{ben2023learning}, the continuous-time formulation offers the advantage that it provides a curve of IS controls over time instead of a discrete set. This allows its application for any time stepping in the IS-TL paths and thereby eliminates the need for ad-hoc interpolations often needed  in the discrete setting.

In the multidimensional setting,  the cost of solving the backward HJB equations increases exponentially with respect to the dimension $d$ (curse of dimensionality). In \cite{ben2023learning}, we proposed a learning-based approach to reduce this effect. In that approach, the value function is approximated using an ansatz function, the parameters of which are learned through a stochastic optimization algorithm (see Figure~\ref{fig:learningscheme} for a schematic illustration of the approach). In this  work, we present an alternative method using a dimension reduction approach for SRNs (see Figure~\ref{fig:schematic} for a schematic illustration of the approach). The proposed methodology is to adapt the MP idea originally introduced in \cite{gyongy1986mimicking} for the setting of diffusion-type stochastic differential equations (SDEs) to the SRN framework, resulting in a significantly lower-dimensional process, preserving the marginal distribution of the original full-dimensional SRN. The propensities characterizing the lower-dimensional MP process can be  approximated using $L^2$ regression. Using the resulting low-dimensional SRN, we derive an approximate value function and, consequently, near-optimal IS controls while reducing the effect of the curse of dimensionality. By mapping the IS controls to the original full-dimensional SRNs, we derive an unbiased IS-MC estimator for the TL scheme. Compared to the learning-based approach presented in \cite{ben2023learning}, this novel MP-IS approach eliminates the need for an ansatz function to model the value function. This approach allows its application to general observables $g$ that differ from indicator functions for rare event estimation, because no prior knowledge regarding the shape of the value function and suitable ansatz functions is required.

To the best of our knowledge, we are the first to establish the MP idea for SRNs and apply it to derive an efficient pathwise IS for MC methods. Initially, the MP idea was introduced for Itô stochastic processes in \cite{krylov1987nonlinear,gyongy1986mimicking} and was later generalized to martingales and semimartingales \cite{kurtz2001stationary,bentata2009mimicking}. In addition, MP has been widely applied for dimension reduction in SDEs \cite{legoll2010effective}, particularly in financial applications \cite{piterbarg2006markovian,djehiche2014risk}. For instance, in \cite{bayer2019implied}, solving  HJB equations for an MP process was pursued but in the setting of Itô SDEs with the application of pricing American options. In \cite{hartmann2016model}, MP was used for control problems and IS problems for rare events in high-dimensional diffusion processes with multiple time scales. In this work, we introduce the general dimension reduction framework of MP for SRNs such that it can be applied to other problems beyond the selected IS application. (e.g., solving the chemical master equation \cite{mikeev2019approximate} or the Kolmogorov backward equations \cite{moraes2015simulation}).

The remainder of this work is organized as follows. Sections~\ref{sec:SRNs}, \ref{sec:exp_tau}, \ref{sec:Monte Carlo (MC) estimator}, and \ref{sec:Importance Sampling} recall the relevant SRN, TL, MC, and IS notations and definitions from \cite{ben2023learning}. Next, Section~\ref{sec:ISOC} reviews the connection between IS and SOC by introducting the IS scheme, the value function, and the corresponding dynamic programming theorem from \cite{ben2023learning} in Section~\ref{sec:Approach Formulation}. Then, Section~\ref{sec:HJB} extends the framework to a continuous-time formulation leading to the continuous-time value function and deriving the corresponding HJB equations.
Section~\ref{sec:MP} presents the MP technique for SRNs and shows how the projected dynamics can be computed using $L^2$ regression. Next, Section~\ref{sec:MP_IS} addresses the curse of dimensionality of high-dimensional SRNs occurring from the optimal IS scheme in Section~\ref{sec:ISOC}  by combining the IS scheme with MP (Section~\ref{sec:MP}) to derive near-optimal IS controls. Finally, Section~\ref{sec:numexp} presents the numerical results for the rare event probability estimation to demonstrate the efficiency of the proposed MP-IS approach compared to a standard TL-MC estimator.  

\begin{figure}[h!]
	\caption{Schematic diagram of the learning-based approach in \cite{ben2023learning}.\label{fig:learningscheme}}
	\includegraphics[trim={0 3cm 1cm 2cm},clip,width=\textwidth]{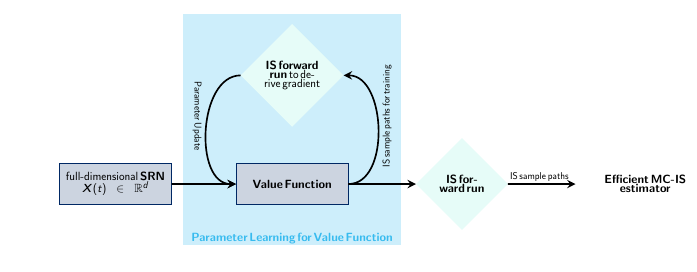}
\end{figure}
\begin{figure}[h!]
	\caption{Schematic diagram of the MP-IS approach. \label{fig:schematic}}
	\includegraphics[trim={0 3cm 1cm 2cm},clip,width=\textwidth]{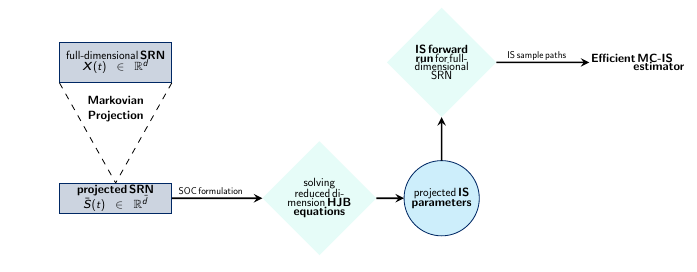}
\end{figure}

\subsection{Stochastic Reaction Networks (SRNs)} %copy of learnin-based approach paper
\label{sec:SRNs}

We recall from \cite{ben2023learning} that an SRN describes the time evolution for a homogeneously mixed chemical reaction system, in which  $d$ distinct species interact through $J$ reaction channels. Each reaction channel  $\mathcal{R}_j$ , $j=1\dots,J$, is given by the relation
\begin{equation}
	\alpha_{j,1} S_1+\dots+\alpha_{j,d} S_d \overset{\theta_j}{\rightarrow}\beta_{j,1} S_1+\dots+\beta_{j,d} S_d,
\end{equation}
where $\alpha_{j,i}$ molecules of species $S_i$ are consumed and $\beta_{j,i}$ molecules are produced. The positive constants $\{\theta_j\}_{j=1}^J$ represent the reaction rates.

This process can be modeled by a Markovian pure jump process,  $\mathbf{X}:[0,T]\times \Omega \to \nset^d$, where ($\Omega$, $\mathcal{F}$, $\mathbb{P}$)  is a probability space. 
We are interested in the time evolution of the state vector, 
\begin{equation}  
	\mathbf{X}(t) = \left(X_1(t), \ldots, X_d(t)\right) \in
	\nset^d ,
\end{equation}
where the $i$-th component,  $X_i(t)$, describes the abundance of the $i$th species present in the system at time $t$. 
The process $\mathbf{X}$  is a continuous-time, discrete-space Markov process characterized by Kurtz's random time change representation \cite{kurtz_2005}:
\begin{equation}
	\label{eq:exact_process}
	\mathbf{X}(t)= \mathbf{x}_{0}+\sum_{j=1}^{J} Y_j \left(\int_0^t  a_{j}(\mathbf{X}(s)) \,  ds \right)  \boldsymbol{\nu}_j,
\end{equation}
where $Y_j:\rset_+{\times} \Omega \to \nset$ are independent unit-rate Poisson processes and the stoichiometric vector is defined as $\boldsymbol{\nu}_j=\left(\beta_{j,1}-\alpha_{j,1},\dots,\beta_{j,d}-\alpha_{j,d}\right) \in \zset^d$. 

For each reaction channel $\mathcal{R}_j$, the propensity function $a_j:\mathbb{N}^d\rightarrow \mathbb{R}_+$ obeys the \textit{non-negativity assumption} (i.e. the system can not attain negative states), which is that $a_j(\mathbf{X})=0$ for $\mathbf{x}$ such that $\mathbf{X}+\mathbf{\nu}_j\notin\mathbb{N}^d$. In our numerical simulations, we consider a propensity derived from the \textit{stochastic mass-action kinetic} principle 
\begin{equation}\label{eq:prop_dynamics}
	a_j(\mathbf{x}):=\theta_j \prod_{i=1}^d \frac{x_i!}{(x_i-\alpha_{j,i})!} \mathbf{1}_{\{x_i\ge \alpha_{j,i}\}}\COMMA
\end{equation}
where $x_i$ is the counting number for species $S_i$. However, the approach presented in this work is not restricted to the particular structure of the propensity function in \eqref{eq:prop_dynamics} (see Remark \ref{rem:prop}).

\subsection{Explicit Tau-Leap Approximation}
\label{sec:exp_tau}
The explicit-TL scheme is a pathwise approximate method based on Kurtz's random time change representation \eqref{eq:exact_process} \cite{gillespie2001approximate,aparicio2001population}. It was originally introduced to overcome the computational drawbacks of exact methods, which become computationally expensive when many reactions fire during a short time interval. 
For a uniform time mesh  $\{t_{0}=0, t_{1},...,t_{N}= T\}$ with step size $\Delta t=\frac{T}{N}$ and a given initial value $\mathbf{X}(0)=\mathbf{x}_0$ , the  explicit-TL approximation for $\mathbf{X}$ is defined by

\begin{align}\label{eq:TL_approx}
	\hat{\mathbf{X}}_0&:= \mathbf{x}_{0}\nonumber\\
	\hat{\mathbf{X}}^{\Delta t}_k&:=\max \left(\textbf{0},\hat{\mathbf{X}}^{\Delta t}_{k-1}+\sum_{j=1}^{J} \mathcal{P}_{k-1,j}\left(a_{j}(\hat{\mathbf{X}}^{\Delta t}_{k-1}) \Delta t\right)  \boldsymbol{\nu}_{j} \right) \COMMA \: 1 \le k \le N,
\end{align}
where $\{\mathcal{P}_{k,j}(r_{k,j})\}_{\{1\leq j\leq J \}}$ are independent Poisson random variables with respective rates $r_{k,j}:=a_{j}(\hat{\mathbf{X}}^{\Delta t}_{k})\Delta t$ conditioned on the current state $\hat{\mathbf{X}}^{\Delta t}_{k}$. The maximum in \eqref{eq:TL_approx} is applied entry-wise. In each TL step, the current state is projected to zero to prevent the process from exiting the lattice (i.e., producing negative values).

\subsection{Biased Monte Carlo estimator}  
\label{sec:Monte Carlo (MC) estimator}
We let $\mathbf{X}$ be an SRN and $g: \rset ^{d} \rightarrow \rset$ be a  scalar observable. 
For a given final time $T$, we estimate $\mathbb{E} \left[g(\mathbf{X}(T))\right]$ using the standard MC-TL estimator: 
\begin{equation}\label{eq:MC}
	\mu_{M} :=\frac{1}{M}\sum_{m=1}^{M} g(\hat{\mathbf{X}}^{\Delta t}_{[m]}(T))\COMMA
\end{equation}
where $\{\hat{\mathbf{X}}^{\Delta t}_{[m]}(T)\}_{m=1}^M$ are independent TL samples. 

The global error for the proposed MC estimator has the following error decomposition:
\begin{align}\label{eq:error_split1}
	\left|\mathbb{E}[g(\mathbf{X}(T))]-\mu_M\right|\leq \underbrace{\left|\mathbb{E}[g(\mathbf{X}(T))]-\mathbb{E}[g(\hat{\mathbf{X}}^{\Delta t}(T))]\right|}_{\text{Bias}}+\underbrace{\left|\mathbb{E}[g(\hat{\mathbf{X}}^{\Delta t}(T))]-\mu_M\right|}_{\text{Statistical Error}}.
\end{align}

Under some assumptions, the TL scheme has a weak order, $\Ordo{\Delta t}$ \cite{li2007analysis}, that is, for sufficiently small $\Delta t$,
\begin{align}\label{eq:weak}
	\left |\mathbb{E} \left[g(\mathbf{X}(T))- g(\hat{\mathbf{X}}^{\Delta t}(T) )\right]\right |\leq C\Delta t
\end{align}
where $C>0$.

The bias and statistical error can be bound equally using $\frac{TOL}{2}$ to achieve the desired accuracy, $\text{TOL}$, with a confidence level of $1-\alpha$ for $\alpha\in (0,1)$, which can be achieved by the step size:
\begin{align}\label{eq:dtstar}
	\Delta t(\text{TOL})= \frac{\text{TOL}}{2\cdot C} 
\end{align}
and 
\begin{align}\label{eq:Mstar1}
	M^*(\text{TOL})=C_{\alpha}^2\frac{4\cdot \text{Var}[g(\hat{\mathbf{X}}^{\Delta t}(T))]}{\text{TOL}^2}
\end{align} 
sample paths, where the constant $C_{\alpha}$ is the $(1-\frac{\alpha}{2})-$quantile for the standard normal distribution. We select $C_{\alpha}=1.96$ for a $95\%$ confidence level corresponding to $\alpha =0.05$.

When estimating rare event probabilities, we are interested in the relative error 
	\begin{align*}
		\frac{	\left|\mathbb{E}[g(\mathbf{X}(T))]-\mu_M\right|}{\left| \mathbb{E}[g(\mathbf{X}(T))]\right|}.
		\end{align*}
	In this context, to achieve a prescribed relative tolerance $\text{TOL}_{rel}$, we use step size
	\begin{align}\label{eq:reldt}
			\Delta t_{rel}(\text{TOL}_{rel})= \frac{\text{TOL}_{rel}\left| \mathbb{E}[g(\mathbf{X}(T))]\right|}{2\cdot C} 
		\end{align}
	and
	\begin{align}\label{eq:relM}
			M^*_{rel}(\text{TOL}_{rel})=C_{\alpha}^2\frac{4\cdot \text{Var}[g(\hat{\mathbf{X}}^{\Delta t}(T))]}{\text{TOL}_{rel}^2\left| \mathbb{E}[g(\mathbf{X}(T))]\right|^2}
	\end{align}
	sample paths.

Given that the computational cost to simulate a single path is $\Ordo{ {\Delta t}^{-1}}$,  the expected total computational complexity is $\Ordo{\text{TOL}^{-3}}$ and $\Ordo{\text{TOL}_{rel}^{-3}}$ for the absolute and relative errors, respectively.

\subsection{Importance Sampling}  
\label{sec:Importance Sampling}
Using IS techniques \cite{kuwahara2008efficient,gillespie2009refining,gillespie2019guided,daigle2011automated,cao2013adaptively,gillespie2019guided,roh2019data} can improve the computational costs for the crude MC estimator through variance reduction in \eqref{eq:Mstar1}. For a general motivation, we refer to \cite{ben2023learning} Section~1.4. For illustrating the IS method, let us consider the general problem of estimating  $\mathbb{E}[g(Y)]$, where $g$ is a given observable and $Y$ is a random variable taking values in $\mathbb{R}$ with the probability density function  $\rho_{Y}$. We let $\hat{\rho}_Z$ be the probability density function for an auxiliary real random variable $Z$. The MC estimator under the IS measure is 
\begin{align}
	\mu_{M}^{IS}=\frac{1}{M} \sum_{j=1}^M L(Z_{[j]})\cdot g(Z_{[j]}),
\end{align}
where $Z_{[j]}$ are independent and identically distributed samples from $\hat{\rho}_Z$ for $j=1,\dots,M$ and the likelihood factor is given by $L(Z_{[j]}):=\frac{\rho_Y(Z_{[j]})}{\hat{\rho}_Z(Z_{[j]})}$. The IS estimator retains the expected value of \eqref{eq:MC} (i.e., $\mathbb{E}[L(Z)g(Z)]=\mathbb{E}[g(Y)]$), but the variance can be reduced due to a different second moment $\mathbb{E}\left[\left(L(Z)\cdot g(Z)\right)^2\right]$.

Determining an auxiliary probability measure that substantially reduces the variance compared with the original measure is challenging and strongly depends on the structure of the considered problem. In addition, the derivation of the new measure must come with a moderate additional computational cost to ensure an efficient IS scheme. This work uses the path-dependent change of probability measure introduced in \cite{ben2023learning}, employing an IS measure derived from changing the Poisson random variable rates in the TL paths. Section~\ref{sec:Approach Formulation} recalls the SOC formulation for optimal IS parameters from \cite{ben2023learning} and extends it with a novel HJB formulation. We conclude this consideration in Section~\ref{sec:MP_IS}, combining the IS scheme with a dimension reduction approach to reduce the computational cost.

\section{Importance Sampling via Stochastic Optimal Control Formulation}\label{sec:ISOC}
\subsection{Dynamic Programming for Importance Sampling Parameters}
\label{sec:Approach Formulation}
This section revisits the connection between optimal IS measure determination within a class of probability measures, and the SOC formulated originally in \cite{ben2023learning}. 
We let $\mathbf{X}$ be an SRN as defined in Section~\ref{sec:SRNs} and let $\hat{\mathbf{X}}^{\Delta t}$ denote its TL approximation as given by \eqref{eq:TL_approx}. Then, the goal is to derive a near-optimal IS measure to estimate $\mathbb{E} \left[g(\mathbf{X}(T))\right]$. We limit ourselves to the parameterized class of IS schemes used in \cite{ben2020importance,ben2023learning}:
\begin{align}\label{eq:path_IS}
	\overline{\mathbf{X}}_{n+1}^{\Delta t}&=\max\left(\textbf{0},\overline{\mathbf{X}}_{n}^{\Delta t}+\sum_{j=1}^J\bar{P}_{n,j}\boldsymbol{\nu}_j\right) ,~~~ n=0,\dots,N-1,\\
	\overline{\mathbf{X}}_{0}^{\Delta t}&=\mathbf{x}_0\nonumber,
\end{align}
where the measure change is obtained by modifying the Poisson random variable rates of the TL paths:
\begin{equation}\label{eq:measure_change}
	\bar{P}_{n,j}=\bar{\mathcal{P}}_{n,j}\left(\delta_{n,j}^{\Delta t}(\overline{\mathbf{X}}^{\Delta t}_n)\Delta t\right),~~~ n=0,\dots, N-1, j=1,\dots,J .
\end{equation}
In \eqref{eq:measure_change}, $\delta_{n,j}^{\Delta t}(\mathbf{x})\in\mathcal{A}_{\mathbf{x},j}$ is the control parameter at time step $n$, under reaction $j$, and in state $\mathbf{x}\in\mathbb{N}^d$. In addition, $\{\bar{\mathcal{P}}_{n,j}(r_{n,j})\}_{\{1\leq j\leq J \}}$ are  independent Poisson random variables, conditioned on $\overline{\mathbf{X}}^{\Delta t}_{n}$,  with the respective rates $r_{n,j}:=\delta_{n,j}^{\Delta t}(\overline{\mathbf{X}}^{\Delta t}_n)\Delta t$. The set of admissible controls is
\begin{align}\label{eq:addmissibleset}
	\mathcal{A}_{\mathbf{x},j}=\begin{cases}
		\{0\}&,\text{if }a_j(\mathbf{x})=0\\
		\{y\in\mathbb{R}: y>0\}&,\text{otherwise}.
	\end{cases}
\end{align} 
In the following, we use the vector notation $\left(\boldsymbol{\delta}_n^{\Delta t}(\mathbf{x})\right)_j:=\delta_{n,j}^{\Delta t}(\mathbf{x}) $ and $\left(\bar{\mathbf{P}}_n\right)_j:=\bar{P}_{n,j}$ for $j=1,\dots,J$.

The corresponding likelihood ratio of the path $\{\overline{\mathbf{X}}^{\Delta t}_n: n=0,\dots,N\}$ for the IS parameters $\boldsymbol{\delta}_n^{\Delta t}(\mathbf{x}) \in \times_{j=1}^J \mathcal{A}_{\mathbf{x},j}$ is 
\begin{equation}\label{eq:likelihood}
	L\left(\left(\bar{\mathbf{P}}_0,\dots,\bar{\mathbf{P}}_{N-1}\right),\left(\boldsymbol{\delta}_0^{\Delta t}(\overline{\mathbf{X}}^{\Delta t}_0),\dots,\boldsymbol{\delta}_{N-1}^{\Delta t}(\overline{\mathbf{X}}^{\Delta t}_{N-1})\right)\right)=\prod_{n=0}^{N-1} L_n(\bar{\mathbf{P}}_n,\boldsymbol{\delta}_n^{\Delta t}(\overline{\mathbf{X}}^{\Delta t}_n)),
\end{equation}
where the likelihood ratio update at time step $n$ is 
\begin{align}\label{eq:stepwiselh}
	L_n(\bar{\mathbf{P}}_n,\boldsymbol{\delta}_n^{\Delta t}(\overline{\mathbf{X}}^{\Delta t}_n))
	&=\prod_{j=1}^J\exp\left(-(a_j(\overline{\mathbf{X}}_{n}^{\Delta t})-\delta_{n,j}^{\Delta t}(\overline{\mathbf{X}}^{\Delta t}_n))\Delta t\right)\left(\frac{a_j(\overline{\mathbf{X}}_{n}^{\Delta t})}{\delta_{n,j}^{\Delta t}(\overline{\mathbf{X}}^{\Delta t}_n)}\right)^{\bar{P}_{n,j}}\nonumber\\
	&=\exp\left(-\left(\sum_{j=1}^J a_j(\overline{\mathbf{X}}_{n}^{\Delta t})-\delta_{n,j}^{\Delta t}(\overline{\mathbf{X}}^{\Delta t}_n)\right)\Delta t\right)  \cdot \prod_{j=1}^J\left(\frac{a_j(\overline{\mathbf{X}}_{n}^{\Delta t})}{\delta_{n,j}^{\Delta t}(\overline{\mathbf{X}}^{\Delta t}_n)}\right)^{\bar{P}_{n,j}}.
\end{align} 
To simplify the notation, we use the convention that  $\frac{a_j(\overline{\mathbf{X}}_{n}^{\Delta t})}{\delta_{n,j}^{\Delta t}(\overline{\mathbf{X}}^{\Delta t}_n)}=1$, whenever both $a_j(\overline{\mathbf{X}}_{n}^{\Delta t})=0$ and $\delta_{n,j}^{\Delta t}(\overline{\mathbf{X}}^{\Delta t}_n)=0$ in (\ref{eq:stepwiselh}). From (\ref{eq:addmissibleset}), this results in a factor of $1$ in the likelihood ratio for reactions where $a_j(\overline{\mathbf{X}}_{n}^{\Delta t})=0$.

Using the introduced change of measure \eqref{eq:measure_change}, the quantity of interest can be expressed with respect to the new measure:
\begin{equation}\label{eq:expis}
	\mathbb{E}[g(\hat{\mathbf{X}}^{\Delta t}_N)]=\mathbb{E}\left[L\left(\left(\bar{\mathbf{P}}_0,\dots,\bar{\mathbf{P}}_{N-1}\right),\left(\boldsymbol{\delta}_0^{\Delta t}(\overline{\mathbf{X}}^{\Delta t}_0),\dots,\boldsymbol{\delta}_{N-1}^{\Delta t}(\overline{\mathbf{X}}^{\Delta t}_{N-1})\right)\right)\cdot g(\overline{\mathbf{X}}^{\Delta t}_N)\right],
\end{equation}
with the expectation on the right-hand side  of \eqref{eq:expis} taken with respect to the dynamics in \eqref{eq:path_IS}. 

Next, we recall the connection between the optimal second moment minimizing IS parameters $\{\boldsymbol{\delta}_n^{\Delta t}(\mathbf{x})\}_{n=0,\dots,N-1; \mathbf{x}\in\mathbb{N}^d}$ and the corresponding discrete-time dynamic programming relation from \cite{ben2023learning}. We revisit the definition of the discrete-time value function $u_{\Delta t}(\cdot,\cdot)$ in Definition~\ref{def:optival2}, allowing the formulation of the dynamic programming equations in Theorem~\ref{theo:exact_optival}. The proof and further details for Theorem~\ref{theo:exact_optival} are provided in \cite{ben2023learning}.

\begin{definition}[Value function]\label{def:optival2}
	For a given $\Delta t>0$, the discrete-time \textit{value function} $u_{\Delta t}(\cdot,\cdot)$ is defined as the optimal (infimum) second moment for the proposed IS estimator. For time step $0 \le n \le N$ and state $\mathbf{x} \in \mathbb{N}^d$, 
	\begin{align}\label{eq:valuefct}
		u_{\Delta t}(n,\mathbf{x})
		&=\inf_{\{\boldsymbol{\delta}^{\Delta t}_k\}_{k=n,\dots,N-1} \in \mathcal{A}^{N-n}}\mathbb{E}\left[g^2\left(\overline{\mathbf{X}}_N^{\Delta t}\right)\prod_{k=n}^{N-1} L_k^2\left(\bar{\mathbf{P}}_k,\boldsymbol{\delta}_k^{\Delta t}(\overline{\mathbf{X}}_k^{\Delta t})\right)\middle| \overline{\mathbf{X}}_n^{\Delta t}=\mathbf{x}\right],
	\end{align}
	where $\mathcal{A}=\bigtimes_{\mathbf{x}\in\mathbb{N}^d}\bigtimes_{j=1}^J\mathcal{A}_{\mathbf{x},j}\in\mathbb{R}^{\mathbb{N}^d \times J}$ is the admissible set for the IS parameters, and $u_{\Delta t}(N,\mathbf{x})=g^2(\mathbf{x})$, for any $\mathbf{x} \in  \mathbb{N}^d$.
\end{definition}
\begin{theorem}[Dynamic programming for IS parameters \cite{ben2023learning}]\label{theo:exact_optival}
	For $\mathbf{x}\in \mathbb{N}^d$ and the given step size $\Delta t>0$, the discrete-time value function $u_{\Delta t}(n,\mathbf{x})$ fulfills the dynamic programming relation:
	\begin{small}
		\begin{align}\label{eq:exact_optival}
			u_{\Delta t}(N,\mathbf{x})&=g^2(\mathbf{x})\nonumber\\
			\text{and for } n&=N-1,\dots,0, \:  \text{and} \:  \mathcal{A}_\mathbf{x}:=\bigtimes_{j=1}^J\mathcal{A}_{\mathbf{x},j},\nonumber\\
			u_{\Delta t}(n,\mathbf{x})&=\inf_{\boldsymbol{\delta}_n^{\Delta t}(\mathbf{x})\in\mathcal{A}_\mathbf{x}}\exp\left(\left(-2\sum_{j=1}^J a_j(\mathbf{x})+\sum_{j=1}^J\delta_{n,j}^{\Delta t}(\mathbf{x})\right)\Delta t\right) \\
			&~~~~~~~~~~~~~\times \sum_{\mathbf{p} \in \mathbb{N}^J}\left(\prod_{j=1}^{J} \frac{(\Delta t \cdot \delta_{n,j}^{\Delta t}(\mathbf{x}))^{p_j}}{p_j!} (\frac{a_j(\mathbf{x})}{\delta_{n,j}^{\Delta t}(\mathbf{x})})^{2p_j} \right)\cdot u_{\Delta t}(n+1,\max(\mathbf{0},\mathbf{x}+ \boldsymbol{\nu}\mathbf{p})),\nonumber
		\end{align}
	\end{small}
	where $\boldsymbol{\nu}=\left(\boldsymbol{\nu}_1, \dots,\boldsymbol{\nu}_J\right)\in\mathbb{Z}^{d\times J}$.
\end{theorem}

Analytically solving the minimization problem \eqref{eq:exact_optival} is challenging due to the infinite sum. In \cite{ben2023learning}, the problem is solved by approximating the value function \eqref{eq:valuefct} using a truncated Taylor expansion of the dynamic programming \eqref{eq:exact_optival}. To overcome the curse of dimensionaliy, a learning-based approach for the value function was proposed. Instead, in this work, we utilize a continuous-time SOC formulation, leading to a set of coupled $d$-dimensional ODEs, the HJB equations (refer to Section~\ref{sec:HJB}). We deal with the curse of dimensionality issue by using a dimension reduction technique, namely the MP, as explained in Section~\ref{sec:MP}.

\subsection{Derivation of Hamilton--Jacobi--Bellman (HJB) Equations}\label{sec:HJB}

In Corollary~\ref{prop:HJB}, the discrete-time dynamic programming relation in Theorem~\ref{theo:exact_optival} is replaced by its analogous continuous-time relation, resulting in a set of ODEs known as the HJB equations. The continuous-time value function $\tilde{u}(\cdot,\mathbf{x}):[0,T]\rightarrow \mathbb{R}$,  $\mathbf{x}\in\mathbb{N}^d$, is the limit of the discrete value function $u_{\Delta t}(\cdot,\mathbf{x})$ as the step size $\Delta t$ approaches zero. In addition, the IS controls $\boldsymbol{\delta}(\cdot,\mathbf{x}): [0,T]\rightarrow \mathcal{A}_{\mathbf{x}}$ become time-continuous curves for $\mathbf{x}\in\mathbb{N}^d$.

\begin{corollary}[HJB equations for IS parameters]\label{prop:HJB}
	For all $\mathbf{x}\in\mathbb{N}^d$, the continuous-time value function $\tilde{u}(t, \mathbf{x})$ fulfills \eqref{eq:HJB} for $t\in [0,T]$: 
	
	\begin{align}\label{eq:HJB}
		\tilde{u}(T, \mathbf{x}) &=g^2(\mathbf{x}) \nonumber \\
		-\frac{d\tilde{u}}{dt} (t, \mathbf{x}) &=\inf _{\boldsymbol{\delta}(t,\mathbf{x}) \in \mathcal{A}_\mathbf{x}}\left(-2 \sum_{j=1}^J a_j(\mathbf{x})+\sum_{j=1}^J \delta_j(t,\mathbf{x})\right) \tilde{u}(t, \mathbf{x})+\sum_{j=1}^J \frac{a_j(\mathbf{x})^2}{\delta_j(t,\mathbf{x})} \tilde{u}\left(t, \max \left(0, \mathbf{x}+\nu_j\right)\right),
	\end{align}
		where $\delta_j(t,\mathbf{x}):=\left(\boldsymbol{\delta}(t,\mathbf{x})\right)_j$.
\end{corollary}

\begin{proof}
	The proof of the corollary is presented in Appendix~\ref{Appendix:HJB}.
\end{proof}

If $\tilde{u}(t,\mathbf{x})>0$ for all $\mathbf{x}\in\mathbb{N}^d$ and $t\in[0,T]$, we can solve the minimization problem in \eqref{eq:HJB} in closed form, such that the optimal controls are given by
\begin{align}\label{eq:deltatilde}
	\tilde{\delta}_{j}(t,\mathbf{x})&= a_j(\mathbf{x})\sqrt{\frac{ \tilde{u}\left(t, \max \left(0, \mathbf{x}+\nu_j\right)\right)}{ \tilde{u}\left(t,\mathbf{x}\right)}}
\end{align}
and \eqref{eq:HJB} simplifies to
\begin{align}\label{eq:HJB2} 
	\frac{d\tilde{u}}{dt} (t, \mathbf{x}) &=-2\sum_{j=1}^J a_j(\mathbf{x})\left( \sqrt{\tilde{u}(t,\mathbf{x})\tilde{u}(t,\max(0,\mathbf{x}+\nu_j))}-\tilde{u}(t,\mathbf{x})\right).
\end{align}

To estimate rare event probabilities with an observable $g(\mathbf{x})=\mathbbm{1}_{\{x_i>\gamma\}}$, we encounter $\tilde{u}(t,\mathbf{x})=0$ for some $\mathbf{x}\in\mathbb{N}^d$; therefore, we modify \eqref{eq:HJB} by approximating the final condition $g(\mathbf{x})$ using a sigmoid: 
\begin{align}\label{eq:sigmoid}
	\tilde{g}(\mathbf{x})=\frac{1}{1+\exp(-b-\beta x_i)} >0
\end{align}
with appropriately chosen parameters $b\in\mathbb{R}$ and $\beta \in \mathbb{R}$. By incorporating the modified final condition, we obtain an approximate value function by solving \eqref{eq:HJB2} using an ODE solver (e.g. \texttt{ode23s} from MATLAB). When using the numerical solver, we truncate the infinite state space $\mathbb{N}^{\bar{d}}$ using sufficiently large upper bounds. The approximated near-optimal IS controls are then expressed by \eqref{eq:deltatilde}. By the truncation of the infinite state space and the approximation of the final condition $g$ by $\tilde{g}$, we introduce a bias to the value function. This can impact the amount of variance reduction in the IS-MC forward run, however, the IS-MC estimator is bias-free.

The cost for the ODE solver scales exponentially with the dimension $d$ of the SRNs, making this approach infeasible for high-dimensional SRNs. Section~\ref{sec:MP} presents a dimension reduction approach for SRNs employed in Section~\ref{sec:MP_IS} to derive suboptimal IS controls for a lower-dimensional SRN. We later demonstrate how these controls are mapped to the full-dimensional SRN system. 

\begin{remark}[Continuous-time IS controls]
	In Corollary~\ref{prop:HJB} and Theorem~\ref{theo:exact_optival}, we present two alternative methods to express the value function \eqref{eq:valuefct} and the IS controls. Utilizing the HJB framework, we can derive continuous controls across time. This allows any time stepping $\Delta t$ in the IS-TL forward run and eliminates the need for ad-hoc interpolations. 
\end{remark}

\section{Markovian Projection for Stochastic Reaction Networks}\label{sec:MP}
\subsection{Formulation}

To address the curse of dimensionality problem when deriving near-optimal IS controls, we project the SRN to a lower-dimensional network while preserving the marginal distribution of the original high-dimensional SRN system. We adapt the MP idea originally introduced in \cite{gyongy1986mimicking} for the setting of diffusion type stochastic differential equations to the SRNs framework. For an $d$-dimensional SRN state vector, $\mathbf{X}(t)$, we introduce a projection to a $\bar{d}$-dimensional space such that $1\le\bar{d}\ll d$:
 \begin{align*}
	P:\mathbb{R}^d\rightarrow \mathbb{R}^{\bar{d}}: \: \mathbf{x}\mapsto \mathbf{P} \mathbf{x},
\end{align*} 
where $\mathbf{P}\in\mathbb{R}^{\bar{d}\times d}$ is a given matrix. This section develops a general MP framework for arbitrary projections with $\bar{d}\geq 1$. However, the choice of the projection depends on the quantity of interest. In particular, when considering rare event probabilities with an observable $g(\mathbf{x})=\mathbbm{1}_{\{x_i>\gamma\}}, \gamma\in\mathbb{R}$ as we do in Section~\ref{sec:MP_IS}, the projection operator is of the form
\begin{align}\label{eq:projection}
	P(\mathbf{x})=\left(0,\dots,\underset{\color{gray}{i-1}}{0}, \underset{\color{gray}{i}}{1},\: \underset{\color{gray}{i+1}}{0},\dots,0  \right) \mathbf{x}.
\end{align}

In the derivation of the MP, we assume that 
\begin{equation}\label{eq:assumpnoninfty}
	\mathbb{E}[a_j(\mathbf{X}(t))\mid P(\mathbf{X}(t))=\mathbf{s},\mathbf{X}(0)=\mathbf{x}_0]<\infty
\end{equation}
for $\mathbf{s}\in\mathbb{N}^{\bar{d}}$, $t\in[0,T]$ and $1\leq j\leq J$.

%The projected process $\boldsymbol{S}(t):=P(\mathbf{X}(t))$, for $t \in [0,T]$, is non-Markovian.
 Theorem~\ref{theo:MP} shows that a  $\bar{d}$ dimensional SRN, $\bar{\boldsymbol{S}}(t)$ exists that follows the same conditional distribution as $\boldsymbol{S}(t):=P(\mathbf{X}(t))$  conditioned on the initial state $\mathbf{X}(0)=\mathbf{x}_0$ for all $t \in [0,T]$. 

\begin{theorem}[MP for SRNs]\label{theo:MP}
	We let $\bar{\boldsymbol{S}}(t)$ be a $\bar{d}$-dimensional stochastic process whose dynamics are given by
	\begin{align}\label{eq:proj}
		\bar{\boldsymbol{S}}(t)= P(\mathbf{x}_{0})+\sum_{j=1}^{J} \bar{Y}_j  \left(\int_0^t  \bar{a}_{j}(\tau,\bar{\boldsymbol{S}}(\tau)) d\tau \right) \underbrace{P(\boldsymbol{\nu}_j)}_{=:\bar{\boldsymbol{\nu}}_j},
	\end{align}
	for $t\in[0,T]$, where $\bar{Y}_j $ denotes independent unit-rate Poisson processes and $\bar{a}_{j}$, $j=1,\dots,J$, are characterized by
	
	\begin{align}\label{eq:abar}
		\bar{a}_j(t,\boldsymbol{s}):=\mathbb{E}\left[a_j(\mathbf{X}(t))\middle|P\left(\mathbf{X}(t)\right)=\boldsymbol{s}, \mathbf{X}(0)=\mathbf{x}_0 \right] \text{, for } 1\le j\le J, \boldsymbol{s}\in \mathbb{N}^{\bar{d}}.
	\end{align}
	Thus, $\boldsymbol{S}(t)\mid_{\{\mathbf{X}(0)=\mathbf{x}_0\}}=P(\mathbf{X}(t))\mid_{\{\mathbf{X}(0)=\mathbf{x}_0\}}$ and $\bar{\boldsymbol{S}}(t)\mid_{\{\mathbf{X}(0)=\mathbf{x}_0\}}$ have the same distribution for all $t\in[0,T]$.
\end{theorem}

\begin{proof}
	The proof for Theorem~\ref{theo:MP} is given in Appendix~\ref{Appendix:MP}.
\end{proof}
In Theorem \ref{theo:MP}, we require the assumption in \eqref{eq:assumpnoninfty} to hold in order to ensure that the MP propensity in \eqref{eq:abar} is well-defined. Assumption  \eqref{eq:assumpnoninfty} does not hold for all SRNs. However, in Remark \ref{rem:ass}, we present sufficient conditions guaranteeing \eqref{eq:assumpnoninfty}.

\begin{remark}[Sufficient conditions for assumption \eqref{eq:assumpnoninfty}]\label{rem:ass}
	If the propensity functions $\{a_j(\cdot)\}_{j=1}^J$ are bounded, i.e., it exist bounds $K_j\in\mathbb{R}^+$ such that for any $\mathbf{x}\in\mathbb{N}^d$ $a_j(\mathbf{x})<K_j$ for $1\leq j\leq J$. Then, the expectation in \eqref{eq:assumpnoninfty} is bounded by the same bounds. An alternative condition is that the set
	\begin{align*}
		\mathcal{K}_{t,\mathbf{s}}:=\left\{ \mathbf{y}\in\mathbb{N}^d: P(\mathbf{y})=\mathbf{s} \quad \text{and} \quad \mathbb{P}(\mathbf{X}(t)=\mathbf{y})>0\right\}
	\end{align*}
	 is finite for all $\mathbf{s}\in\mathbb{N}^{\bar{d}}$ and $t\in[0,T]$. Consequently, the expectation in \eqref{eq:assumpnoninfty} corresponds to a finite sum and is finite.
	 
	 An in-depth analysis is required to establish sharp conditions under which \eqref{eq:assumpnoninfty}  holds. Our numerical analysis show that the MP propensity is well defined in many examples.
\end{remark}

The mimicking process $\{\bar{\boldsymbol{S}}(t)\}_{t\in[0,T]}$ is a SRN and consequently Markovian, whereas the projected full-dimensional process $\{\boldsymbol{S}(t)\}_{t\in[0,T]}$ is non-Markovian.
  	Moreover, the propensities of the full-dimensional process $\{a_j\}_{j=1}^J$ are time-homogeneous functions of the state, whereas the resulting propensities, $\{\bar{a}_j\}_{j=1}^J$ of the MP-SRN $\bar{\boldsymbol{S}}$ are time-dependent (see \eqref{eq:abar}).
Reactions with $P(\boldsymbol{\nu}_j)=0$ do not contribute to the MP propensity in \eqref{eq:abar}. For reactions with $P(\boldsymbol{\nu}_j) \neq 0$, it may occur that their corresponding projected propensity is known analytically. We denote the index set of reactions requiring an estimation of \eqref{eq:abar} (e.g., via a $L^2$ regression as described in Section~\ref{sec:L2}) by $\mathcal{J}_{MP}$. This index set is described as follows:

\begin{align}\label{eq:JMP}
	\mathcal{J}_{MP}:=\left\{1\leq j \leq J: \quad P(\boldsymbol{\nu}_j)\neq0 \quad \text{and}\quad \underbrace{a_j(\mathbf{x})\neq f(P(\mathbf{x}))\text{ for all functions } f:\mathbb{R}^{\bar{d}}\rightarrow \mathbb{R} }_{(*)} \right\},
\end{align}
where condition (*) excludes reaction channels for which the MP propensity is only dependent on $s$ and given in closed form by $\bar{a}_j(t,s)=f(s)$ for the function $f$.

\subsection{Discrete $L^2$ Regression for Approximating Projected Propensities}\label{sec:L2}

To approximate the Markovian propensity $\bar{a}_j$ for $j\in \mathcal{J}_{MP}$, we reformulate \eqref{eq:abar} as a minimization problem and then use discrete $L^2$ regression as described below.

We let $V:=\left\{f:[0,T]\times \mathbb{R}^{\bar{d}}\rightarrow \mathbb{R}: \int_0^T\mathbb{E}[f(t,P(\mathbf{X}(t)))^2]dt<\infty\right\}$. Then, the projected propensities via the MP for $j\in \mathcal{J}_{MP}$ are approximated by
\begin{align}\label{eq:V_t2}
	\bar{a}_j(\cdot,\cdot)&=\text{argmin}_{h\in V}\int_0^T\mathbb{E}\left[\left( a_j(\mathbf{X}(t))-h(t,P(\mathbf{X}(t)))\right)^2\right]dt\nonumber \\
	&\approx\text{argmin}_{h\in V}\mathbb{E}\left[\frac{1}{{N}}\sum_{{n=0}}^{N-1}\left( a_j(\hat{\mathbf{X}}^{\Delta t}_n)-h(t_n,P(\hat{\mathbf{X}}^{\Delta t}_n))\right)^2\right]\nonumber \\
	&\approx\text{argmin}_{h\in V} \frac{1}{M} \sum_{m=1}^M\frac{1}{{N}}\sum_{{n=0}}^{N-1}\left( a_j(\hat{\mathbf{X}}^{\Delta t}_{[m],n})-h(t_n,P(\hat{\mathbf{X}}^{\Delta t}_{[m],n}))\right)^2  ,
\end{align}
where $\left\{\hat{\mathbf{X}}^{\Delta t}_{[m]}\right\}_{m=1}^M$ are $M$ independent TL paths with a uniform time grid $0=t_0<t_1<\dots<t_N=T$ with step size $\Delta t$.

To solve \eqref{eq:V_t2}, we use a discrete $L^2$ regression approach. For the case $\bar{d}=1$, we employ a set of basis functions in $V$, $\{\phi_p(\cdot,\cdot)\}_{p\in\Lambda}$, where $\Lambda\subset \mathbb{N}^{2}$ is a finite index set. In Remark~\ref{rem:orth_basis}, we provide more details on the choice of the basis. Consequently, the projected propensities via MP are approximated by
\begin{align}\label{eq:bara}
	\bar{a}_j(t,s)\approx \sum_{p\in\Lambda}c_p^{(j)}{\phi}_p(t,s), j \in \mathcal{J}_{MP}
\end{align}
where the coefficients $c_p^{(j)}$ must be derived for $j\in \mathcal{J}_{MP}$ and $p\in\Lambda$.

Next, we derive the linear systems of equations, solved by $\{c_p^{(j)}\}_{p\in\Lambda}$ from \eqref{eq:bara} for $j\in \mathcal{J}_{MP}$.
For a given one-dimensional indexing of $\{1,\dots,M\} \times \{0,\dots,N-1\}$, the corresponding design matrix $\boldsymbol{D}\in\mathbb{R}^{MN\times|\Lambda|}$ is given by
\begin{align*}
	D_{k,p}={\phi}_p(t_n,P(\hat{\mathbf{X}}^{\Delta t}_{[m],n})), \text{ for } k=(m,n)\in\{1,\dots,M\} \times \{0,\dots,N-1\}, \quad p\in \Lambda.
\end{align*}
Further, we set  $\psi_{k}^{(j)}=a_j(\hat{\mathbf{X}}^{\Delta t}_{[m],n})$ ($\boldsymbol{\psi}^{(j)}\in \mathbb{R}^{MN}$) for $k\in\{1,\dots,M\} \times \{0,\dots,N-1\}$, and $j\in \mathcal{J}_{MP}$.

Then, the minimization problem in \eqref{eq:V_t2} becomes
\begin{align*}
	\textbf{c}^{(j)}& =\text{argmin}_{\{c_{p}\}_{p\in\Lambda}} \frac{1}{MN} \sum_{m=1}^M\sum_{n=0}^{N-1}\left( a_j(\hat{\mathbf{X}}^{\Delta t}_{[m],n})-\sum_{p\in\Lambda} c_{p} {\phi}\left(t_{n},P(\hat{\mathbf{X}}^{\Delta t}_{[m],n})\right)\right)^2\\
	&=\text{argmin}_{\mathbf{c}\in \mathbb{R}^{\#\Lambda}} \left(\boldsymbol{\psi}^{(j)}-\boldsymbol{D}\textbf{c} \right)^\top\left(\boldsymbol{\psi}^{(j)}-\boldsymbol{D} \textbf{c} \right)\\
	&= \text{argmin}_{\mathbf{c}\in  \mathbb{R}^{\#\Lambda}}\underbrace{{ \boldsymbol{\psi}^{(j)}}^\top\boldsymbol{\psi}^{(j)}-2\mathbf{c}^\top \boldsymbol{D}^\top \boldsymbol{\psi}^{(j)} +\mathbf{c}^\top \boldsymbol{D}^\top \boldsymbol{D}\mathbf{c}}_{=:I(\mathbf{c})}.
\end{align*}

We minimize $I(\mathbf{c})$ with respect to $\mathbf{c}$ by solving
\begin{align*}
	\frac{\partial I(\mathbf{c})}{\partial \mathbf{c}} = -2{\boldsymbol{D} }^\top \boldsymbol{\psi}^{(j)} +2{\boldsymbol{D} }^\top \boldsymbol{\boldsymbol{D} } \mathbf{c}=0
\end{align*}
and obtain the normal equation for $j\in \mathcal{J}_{MP}$:
\begin{align}\label{eq:normal}
	\left(\boldsymbol{D} ^\top \boldsymbol{D} \right)\mathbf{c}^{(j)}=\boldsymbol{D} ^\top\boldsymbol{\psi}^{(j)}.
\end{align}

\begin{remark}[Orthonormal basis approach via empirical inner product]\label{rem:orth_basis}
For the case $\bar{d}=1$, the normal equation with a set of polynomials $\{\phi_p\}_{p\in\Lambda}$ in $\mathbb{R}^2$ can be used to derive the MP propensity $\bar{a}_j$ for $ j\in \mathcal{J}_{MP}$. We use the standard basis $\left\{\phi_{(i_1,i_2)}\right\}_{(i_1,i_2)\in\Lambda}$ for a two-dimensional index set $\Lambda$, where
	$$
	\phi_{(i_1,i_2)}: \mathbb{R}^{2} \rightarrow \mathbb{R}, (t,x) \mapsto t^{i_1}x^{i_2}.
	$$
	
	For better stability \cite{cohen2013stability}, we use the Gram--Schmidt orthogonalization algorithm to determine an orthonormal set of functions for the empirical scalar product: 
	\begin{align}\label{eq:skalar}
		\left\langle\phi_{i}, \phi_{j}\right\rangle_{{M}}=\frac{1}{N} \sum_{n=0}^{N-1} \frac{1}{{M}} \sum_{m=1}^{{M}} \phi_{i}\left(t_n, P \hat{\mathbf{X}}^{\Delta t}_{[m],n}\right) \phi_{j}\left(t_{n},P \hat{\mathbf{X}}^{\Delta t}_{[m],n}\right)
	\end{align}
	to find an orthonormal set of functions.
	We base the empirical scalar product and the normal equation \eqref{eq:normal} on the same set of TL paths, $\{\hat{\mathbf{X}}^{\Delta t}_{[m]}\}_{m=1,\dots,M}$, such that the matrix condition number becomes $\text{cond}(\boldsymbol{D} ^\top \boldsymbol{D})=1$ and $\boldsymbol{D}^\top \boldsymbol{D}=\text{diag}\left(\frac{T}{\Delta t}M,\dots,\frac{T}{\Delta t}M\right)$ \cite{cohen2013stability}.
\end{remark}

	\begin{remark}[Propensity functions]\label{rem:prop}
 The proof of Theorem \eqref{theo:MP} and the derivation of the normal equation \eqref{eq:normal} is general with respect to the structure of the propensity function. The Markovian projection can be applied to SRNs with arbitrary propensity functions, i.e., not restricted to the mass-action-kinetics principle introduced in \eqref{eq:prop_dynamics} (e.g. Hill-type reaction rate law \cite{weiss1997hill}). The choice of suitable basis functions $\{\phi_p\}_{p\in\Lambda}$ for the $L^2$ regression depends on the given example and also on the type of propensity.
	\end{remark}

\subsection{Computational Cost of Markovian Projection}\label{sec:comp_cost_MP}

The computational work to derive an MP for an SRN with $J$ reactions based on a time stepping $\Delta t$ based on $M$ TL paths and an orthonormal set of polynomials (see Remark~\ref{rem:orth_basis}) of size $\#\Lambda$ splits into three types of costs:
\begin{align}\label{eq:MPcost}
	W_{MP}(\#\Lambda,\Delta t,M)\approx M\cdot W_{TL}(\Delta t) + W_{G-S}(\#\Lambda,\Delta t,M)+W_{L^2}(\#\Lambda,\Delta t,M),
\end{align}
where $W_{TL}$, $W_{G-S}$, and $W_{L^2}$ denote the computational costs to simulate a TL path, derive an orthonormal basis (as described in Remark~\ref{rem:orth_basis}), and derive and solve the normal equation in \eqref{eq:normal}, respectively. The dominant terms of these costs contribute as follows:
\begin{align*}
	W_{TL}(\Delta t) & \approx \frac{T}{\Delta t} \cdot J \cdot C_{Poi},\\
	W_{G-S}(\#\Lambda,\Delta t,M)& \approx M \cdot \frac{T}{\Delta t}\cdot  \left(\#\Lambda\right)^3,\\
	W_{L^2}(\#\Lambda,\Delta t,M)& \approx M\cdot \frac{T}{\Delta t}\cdot \left(\left(\#\Lambda\right)^2+\#\mathcal{J}_{MP}\cdot\#\Lambda\right),
\end{align*}
where $C_{Poi}$ represents the cost to simulate one realization of a Poisson random variable. The main computational cost results from deriving an orthonormal basis (see Remark~\ref{rem:orth_basis}). A more detailed derivation of the cost terms is provided in Appendix~\ref{appendix:cost}. For many applications, such as the MP-IS approach presented in Section~\ref{sec:MP_IS}, the MP must be computed only once, such that the computational cost $W_{MP}(\#\Lambda,\Delta t,M)$ can be regarded as an off-line cost.

	\begin{remark}[Simulating MP paths]\label{rem:simMP}
		The MP-SRN $\bar{\mathbf{S}}(t)$ can be simulated as a SRN with inhomogeneous propensity function. The explicit TL scheme in \eqref{eq:TL_approx} can be naturally adapted to this setting. However, we note that the IS approach presented in Section \ref{sec:MP_IS} does not require explicitly simulating paths of the MP-SRN. 
\end{remark}

\section{Importance Sampling for Higher-dimensional Stochastic Reaction Networks via Markovian Projection}\label{sec:MP_IS}
Next, we employ MP to overcome the curse of dimensionality when deriving IS controls from solving \eqref{eq:HJB}. Specifically, we solve the HJB equations in \eqref{eq:HJB2} for a reduced-dimensional MP system (refer to Figure \ref{fig:schematic} for a schematic illustration of the approach). Given a suitable projection $P:\mathbb{R}^d\rightarrow \mathbb{R}^{\bar{d}}$ and a corresponding final condition $\tilde{g}:\mathbb{N}^{\bar{d}}\rightarrow \mathbb{R}$ with $\tilde{g}(P(\mathbf{x}))=g(\mathbf{x})$, the HJB equations \eqref{eq:HJB2} for the MP process are

\begin{align}\label{eq:HJBdbar}
	\tilde{u}_{\bar{d}}(T, \boldsymbol{s}) &=\tilde{g}^2(\boldsymbol{s}), \quad  \boldsymbol{s} \in \mathbb{N}^{\bar{d}} \nonumber \\
	\frac{d\tilde{u}_{\bar{d}}}{dt}(t, \boldsymbol{s}) &=-2\sum_{j=1}^J \bar{a}_j(t,\boldsymbol{s})\left( \sqrt{\tilde{u}_{\bar{d}}(t,\boldsymbol{s})\tilde{u}_{\bar{d}}(t,\max(0,\boldsymbol{s}+\bar{\boldsymbol{\nu}}_j))}-\tilde{u}_{\bar{d}}(t,\boldsymbol{s})\right), \quad t\in[0,T], \boldsymbol{s} \in \mathbb{N}^{\bar{d}}.
\end{align}

For observables of the type $g(\mathbf{x})=\mathbbm{1}_{\{x_i>\gamma\}}$, we use an  MP to a ($\bar{d}=1$)-dimensional process via projection \eqref{eq:projection}, and the final condition is approximated by a positive sigmoid (see \eqref{eq:sigmoid}). 
The solution of \eqref{eq:HJBdbar} is the value function $\tilde{u}_{\bar{d}}$ of the $\bar{d}$-dimensional MP process. 
To obtain continuous-time IS controls for the $d$-dimensional SRN, we substitute the value function $\tilde{u}\left(t,\mathbf{x}\right)$ of the full-dimensional process in \eqref{eq:deltatilde} with the value function $\tilde{u}_{\bar{d}}(t,P(\mathbf{x}))$ of the MP-SRN: 
\begin{align}\label{eq:blue_approach}
	\bar{\delta}_j(t,\mathbf{x})={a_j(\mathbf{x})}\sqrt{\frac{ \tilde{u}_{\bar{d}}\left(t, \max \left(0, P(\mathbf{x}+\boldsymbol{\nu}_j)\right)\right)}{ \tilde{u}_{\bar{d}}\left(t,P(\mathbf{x})\right)}} \text{ for } \mathbf{x}\in\mathbb{N}^d, t\in[0,T].
\end{align}

\begin{remark}[Alternative MP-IS approach]
	In the presented approach, we map the value function of the $\bar{d}$-dimensional MP process to the full-dimensional SRNs. Alternatively, one could also map the optimal controls from the $\bar{d}$-dimensional MP-SRN to the full-dimensional SRNs, leading to the following controls: 
	\begin{align}\label{eq:red_approach}
		\tilde{\delta}^{\bar{d}}_j(t,\mathbf{x})=\bar{a}_j(t,P(\mathbf{x}))\sqrt{\frac{ \tilde{u}_{\bar{d}}\left(t, \max \left(0, P(\mathbf{x}+\boldsymbol{\nu}_j)\right)\right)}{ \tilde{u}_{\bar{d}}\left(t,P(\mathbf{x})\right)}}, \text{ for } \mathbf{x}\in\mathbb{N}^d, t\in[0,T].
	\end{align}
	The numerical experiments demonstrate that this approach results in a comparable variance reduction to the approach presented in \eqref{eq:blue_approach}.
\end{remark}

\begin{remark}[Adaptive MP for $\bar{d}>1$]
 In \eqref{eq:blue_approach}, when utilizing $\tilde{u}_{\bar{d}}$ as the value function for the $d$-dimensional control, we introduce a bias to the optimal IS controls by approximating $\tilde{u}(t,\mathbf{x})$ by $\tilde{u}_{\bar{d}}(t,P(\mathbf{x}))$ for $\mathbf{x}\in\mathbb{N}^d$ and $t\in[0,T]$. For the case $\bar{d}=d$, we have $\tilde{u}_{\bar{d}}(t,P(\mathbf{x}))=\tilde{u}(t,\mathbf{x})$ and the MP produces the optimal IS control for the full-dimensional SRNs. For  $\bar{d}<d$, this equality does not hold, since the interaction (correlation effects) between non-projected species are not taken into account in the MP SRNs, because the MP only ensures that the marginal distributions of $P(\mathbf{X}(t))\mid_{\{\mathbf{X}(0)=\mathbf{x}_0\}}$ and $\bar{\boldsymbol{S}}(t)\mid_{\{\mathbf{X}(0)=\mathbf{x}_0\}}$ are identical. This can be seen in examples in which reactions occur with $P(\boldsymbol{\nu}_j)=0$. Those reactions, are not present in the MP and; thus, are not included in the IS scheme. For the extreme case, $\bar{d}=1$, we expect to achieve the least variance reduction which could be already substantial and satisfactory for many examples as we show in our numerical experiments.
	
	However, examples could exist where a projection to dimension $\bar{d}=1$ is insufficient to achieve a desired variance reduction. In this case, we can adaptively choose a better projection with  increased dimension $\bar{d}=1,2,\dots$ until a sufficient variance reduction is achieved. This will imply an increased computational cost in the MP and in solving the HJB equations \eqref{eq:HJBdbar} for $\mathbf{x}\in\mathbb{N}^{\bar{d}}$. Investigating the effect of $\bar{d}$ on improving the variance reduction of our approach is left for a future work.
\end{remark}

To derive an MP-IS-MC estimator for a given uniform time grid $0=t_0\leq t_1\leq \dots \leq t_N=T$ with step size $\Delta t$, we generate IS paths using the scheme in \eqref{eq:path_IS} with IS control parameters $\delta_{n,j}^{\Delta t}(\mathbf{x})=\bar{\delta}_j(t_n,\mathbf{x})$, as in \eqref{eq:blue_approach}, for $j=1,\dots,J, \mathbf{x}\in\mathbb{R}^d, n=0,\dots,N-1$. Figure~\ref{fig:scheme} presents a schematic illustration of the entire derivation of the MP-IS-MC estimator.

This computational work consists of three cost contributions:
\begin{align}
	W_{MP-IS-MC}&(\#\Lambda,\Delta t,M,M_{fw})\\
	&\approx W_{MP}(\#\Lambda,\Delta t,M)+W_{HJB}(\#\Lambda)+W_{forward}(\Delta t, M_{fw}),\nonumber
\end{align}

where $W_{MP}(\#\Lambda,\Delta t,M)$ denotes the off-line cost to derive the MP (see \eqref{eq:MPcost}), $W_{HJB}(\#\Lambda)$ represents the cost to solve the HJB \eqref{eq:HJBdbar} for the $\bar{d}$-dimensional MP-SRN, and  $W_{forward}(\Delta t, M_{fw})$ indicates the cost of deriving $M_{fw}$ IS paths. The cost to solve the HJB \eqref{eq:HJBdbar} $W_{HJB}(\#\Lambda)$ depends on the used solver, and the cost for the forward run has the following dominant terms:
\begin{align}
	W_{forward}(\Delta t, M_{fw})\approx M_{fw}\cdot\frac{T}{\Delta t}\cdot (J \cdot C_{Poi}+C_{lik}+\#\mathcal{J}_{MP}\cdot C_{\delta}),
\end{align}
where $C_{\delta}$ is the cost to evaluate \eqref{eq:blue_approach}.

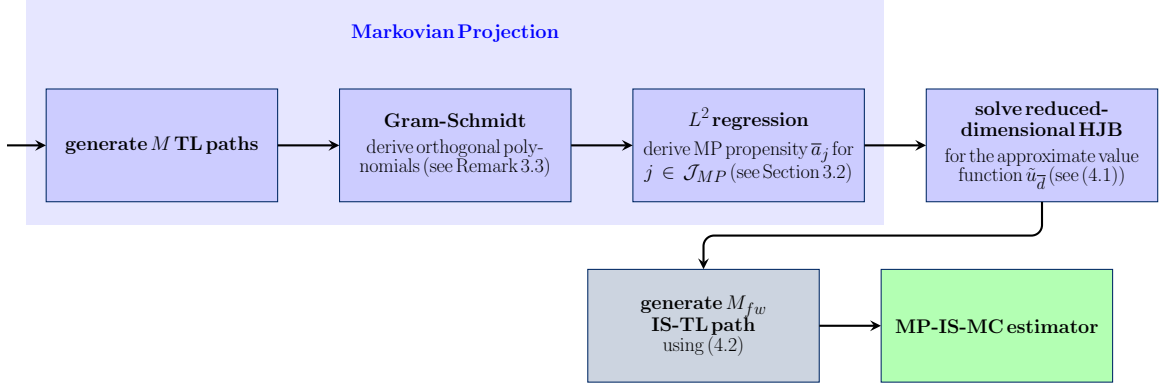
\begin{figure}[h]
	\caption{Schematic diagram MP-IS-MC. The costs of the operations in the first line (blue boxes) are off-line. \label{fig:scheme}}
	\begin{NoHyper}
	\begin{tikzpicture}[scale=0.3, transform shape]
			\node(MPbox)[groupingblue]{};
		\node(start2)[text centered, left of =MPbox, xshift=-19cm, yshift=-1.5cm] {};
		\node (start1) [box,right of =start2, xshift=6cm, yshift=0cm] {\Huge \textbf{generate $M$ TL paths}};
	
		\node (start) [box, right of=start1,xshift=12cm] {{\Huge \textbf{Gram-Schmidt}}\\
			\vspace{0.5cm}
			\Huge derive orthogonal polynomials (see Remark~\ref{rem:orth_basis})};
		\node (MPa) [box, right of=start,xshift=12cm] {\Huge \textbf{$L^2$ regression}\\ \vspace{0.5cm}\Huge derive MP propensity $\bar{a}_j$ for $j\in \mathcal{J}_{MP}$ (see Section~\ref{sec:L2})};
		\node (u) [box, right of=MPa,xshift=12cm] {{\Huge \textbf{solve reduced-}\\ \vspace{0.3cm} \textbf{dimensional HJB}}\\ \vspace{0.5cm}\Huge for the approximate value function $\tilde{u}_{\bar{d}}$ (see \eqref{eq:HJBdbar})};
		\node(blue1) [box, below of=u,yshift=-7cm, xshift=-15cm, draw=helmholtzdarkblue, fill=helmholtzdarkblue!20] {{\Huge \textbf{generate $M_{fw}$ IS-TL path}}\\ \Huge using \eqref{eq:blue_approach}};
	
		\node(MC)[box,right of=blue1, xshift = 12cm, fill=green!30] {\Huge \textbf{MP-IS-MC estimator}};
		\node(MPlabel)[text centered, above of=start, yshift=4cm,text width=10cm]  {\Huge \color{blue}\textbf{Markovian Projection} };
		
		\draw [arrow] (start2) -- (start1);
		\draw [arrow] (start1) -- (start);
		\draw [arrow] (start) -- (MPa);
		\draw [arrow] (MPa) -- (u);
		
		%\draw [arrow] (u.south)  to [out=270,in=180] (blue1.west);
		\draw[rounded corners, arrow] 
		(u.south) |- ($ (MC)!0.5!(MPa) $)
		-| (blue1.north);
		
		\draw [arrow] (blue1) -- (MC);
	\end{tikzpicture}
\end{NoHyper}
\end{figure}

\begin{remark}[Further applications of MP in the context of SRNs]
	In this work, we use the described MP for dimension reduction to derive a sub-optimal change of measure for IS, but the same MP framework can be used for other applications, such as solving the chemical master equation \cite{mikeev2019approximate} or the Kolmogorov backward equations \cite{moraes2015simulation}. We intend to explore these directions in a future work.
\end{remark}

\section{Numerical Experiments and Results}\label{sec:numexp}
Through Examples~\ref{exp:mm} and \ref{exp:Goutsias}, we demonstrate the advantages of the proposed MP-IS approach compared with the standard MC approach for rare event estimations. We numerically demonstrate that the proposed approach achieves a substantial variance reduction compared with standard MC estimators when applied to SRNs with various dimensions.

	\begin{example}[Michaelis--Menten enzyme kinetics \cite{rao2003stochastic}] 
		\label{exp:mm}
		The Michaelis-Menten enzyme kinetics are enzyme-catalyzed reactions describing the interaction of an enzyme $E$  with a substrate $S$, resulting in a product $P$:
		\begin{align*}
			E+S\overset{\theta_1}{\rightarrow} C, ~~C\overset{\theta_2}{\rightarrow} E+S,~~
			C\overset{\theta_3}{\rightarrow} E+P,
		\end{align*}
		where $\theta = (0.001,0.005,0.01)^\top$. 
		We consider the initial state $\mathbf{X}_0=(E(0),S(0),C(0),P(0))^\top=(100, 100, 0, 0)^\top$ and the final time $T=1$. The corresponding propensity and the stoichiometric matrix are given by
		\begin{align*}
			a(\textbf{x})=\left(\begin{array}{c}
				\theta_{1} E S \\
				\theta_{2} C \\
				\theta_{3} C
			\end{array}\right), \quad
			\boldsymbol{\nu}=\left(\begin{array}{ccc}
				-1 & 1 & 1  \\
				-1 & 1 & 0  \\
				1 & -1& -1 \\
				0& 0& 1
			\end{array}\right).
		\end{align*}
		The observable of interest is $g(\mathbf{x})=\mathbf{1}_{\{x_3>22\}}$.
	\end{example}

\begin{example}[Goutsias's model of regulated transcription \cite{goutsias2005quasiequilibrium,kang_separation_2013}]
	\label{exp:Goutsias}
	%We consider the following model of transcription regulation introduced in Goutsias (2005) and studied further in Macnamara, Burrage and Sidje (2007).
	
	The model describes a transcription regulation through the following six molecules:\\
	\begin{tabular}{l l l l} 
		Protein monomer &($M$),&Transcription factor& ($D$),\\
		mRNA &($RNA$),&Unbound DNA &($DNA$),\\
		DNA bound at one site &($DNA\cdot D$),&DNA bounded at two sites &($DNA\cdot 2D$).
	\end{tabular}
	
	These species interact through the following 10 reaction channels
	
	\begin{tabular}{r l r l} 
		R N A & $\overset{\theta_1}{\rightarrow} R N A+M,$ &M & $\overset{\theta_2}{\rightarrow} \varnothing,$ \\
		D N A $\cdot$ D &$\overset{\theta_3}{\rightarrow} R N A+D N A \cdot D, $&
		R N A & $\overset{\theta_4}{\rightarrow} \varnothing,$ \\
		D N A+D & $\overset{\theta_5}{\rightarrow} D N A \cdot D,$ &
		D N A $\cdot$ D &$\overset{\theta_6}{\rightarrow}D N A+D,$ \\
		D N A $\cdot$ D+D & $\overset{\theta_7}{\rightarrow} D N A \cdot 2 D,$ &
		D N A $\cdot$ 2 D & $\overset{\theta_8}{\rightarrow} D N A \cdot D+D,$ \\
		2M & $\overset{\theta_9}{\rightarrow} D,$ &
		D & $\overset{\theta_{10}}{\rightarrow} 2 M$,
	\end{tabular}

	where $(\theta_1,\dots,\theta_{10})=(0.043, 0.0007, 0.0715, 0.0039, 0.0199, 0.479, 0.000199, 8.77\times10^{-12}, 0.083, 0.5)$. As the initial state, we use $\mathbf{X}_0=(M(0),D(0),RNA(0),DNA(0),DNA\cdot D(0),DNA\cdot2D(0))=(2,6,0,0,2,0)$, and the final time is $T=1$. We aim to estimate the rare event probability $\mathbb{P}(D(T)>8)$.
\end{example}

\subsection{Markovian Projection Results}
Through simulations for Examples~\ref{exp:mm} and \ref{exp:Goutsias}, we numerically demonstrate that the distribution of the MP process $\bar{\boldsymbol{S}}(T)\mid_{\{\mathbf{X}_0=\mathbf{x}_0\}}$ matches the conditional distribution of the projected process $\boldsymbol{S}(t)\mid_{\{\mathbf{X}_0=\mathbf{x}_0\}}=P(\mathbf{X}(t))\mid_{\{\mathbf{X}_0=\mathbf{x}_0\}}$, as shown in Theorem~\ref{theo:MP}. For both examples, we use an MP projection with $\bar{d}=1$ using the projection given in \eqref{eq:projection}, where the projected species is indexed as $i=3$ in Example~\ref{exp:mm} and as $i=2$ in Example~\ref{exp:Goutsias}.

The MP is based on $M=10^4$ TL sample paths with a step size of $\Delta t=2^{-8}$ and uses the orthonormal basis of polynomials described in Remark~\ref{rem:orth_basis} with $\Lambda = \{0,1,2\} \times \{0,1,2\}$  for the $L^2$ regression. For the numerical comparison of the full-dimensional SRN with the MP-SRN, we use explicit TL paths with step size $\Delta t=2^{-8}$ (see Remark \ref{rem:simMP}). In Figure \ref{fig:samplepaths}, we show sample paths of the original SRN and the mimicking MP-SRN. Figure~\ref{fig:hist} shows the relative occurrences of states at final time $T$ with $M_{fw}=10^4$ sample paths, comparing the TL estimate of $P(\mathbf{X}(t))\mid_{\{\mathbf{X}_0=\mathbf{x}_0\}}$ and the TL-MP estimate of $\bar{\boldsymbol{S}}(T)\mid_{\{\mathbf{X}_0=\mathbf{x}_0\}}$. In both examples, the one-dimensional MP process mimics the distribution of the state of interest $X_i(T)$ of the original SRNs. Further quantification and analyses of the MP error are left for future work. In this work, a detailed analysis of the MP error is less relevant because the MP is used as a tool to derive IS controls for the full-dimensional process, and the IS is bias-free with respect to the TL scheme.
	\begin{figure}
	\caption{Sample paths of TL and MP-TL for a step size of $\Delta t=2^{-8}$.\label{fig:samplepaths}}
	\subfloat [Example \ref{exp:mm}: 30 TL paths] 
	{\includegraphics[scale=0.40]{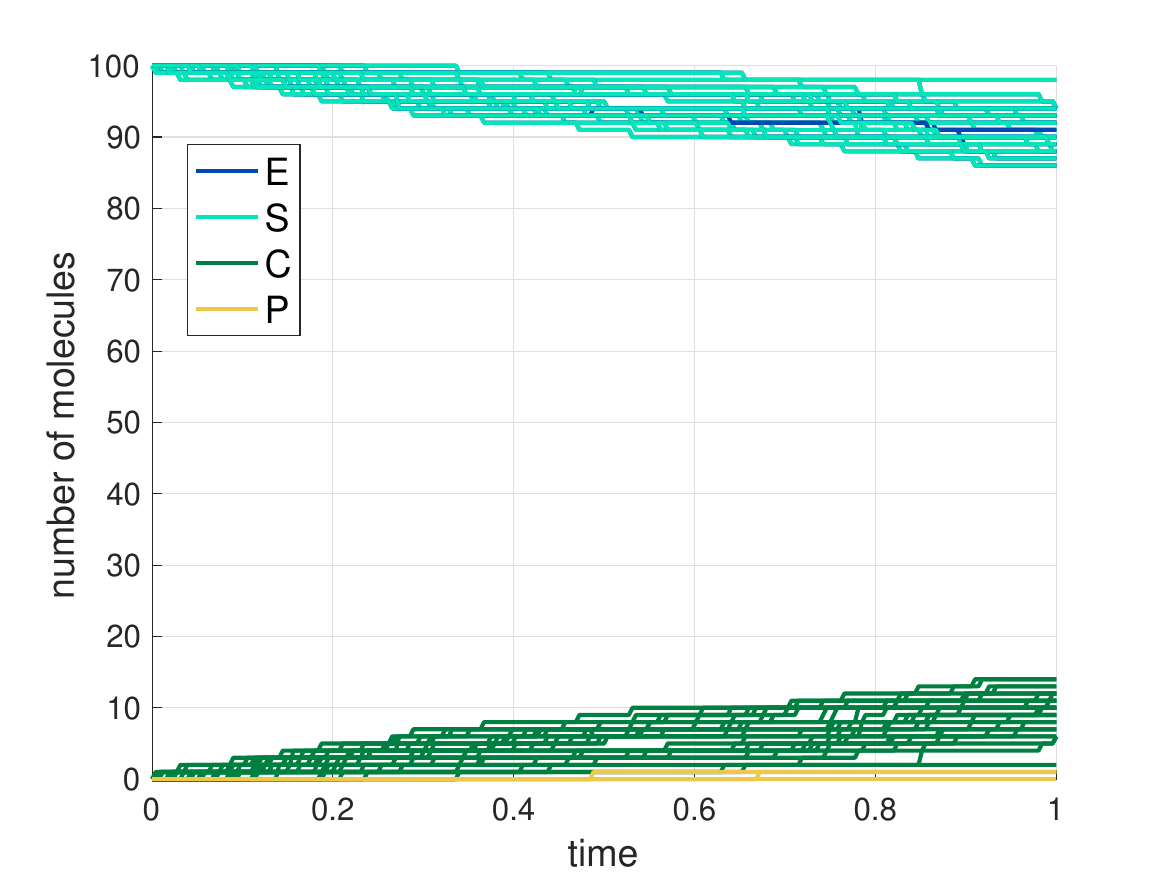}}
	\subfloat [Example \ref{exp:mm}: 30 TL-MP paths] 
	{\includegraphics[scale=0.40]{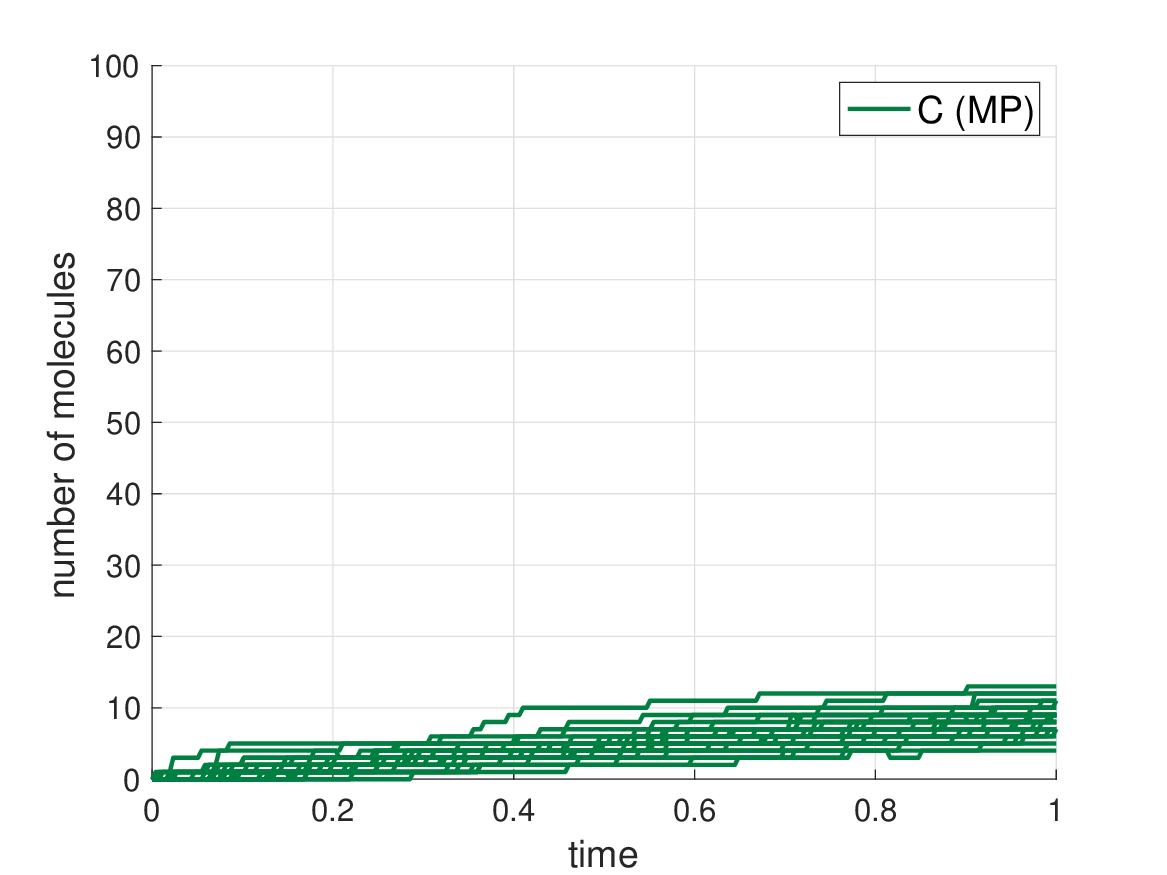}}\\
	\subfloat [Example \ref{exp:Goutsias}: 30 TL paths] 
	{\includegraphics[scale=0.40]{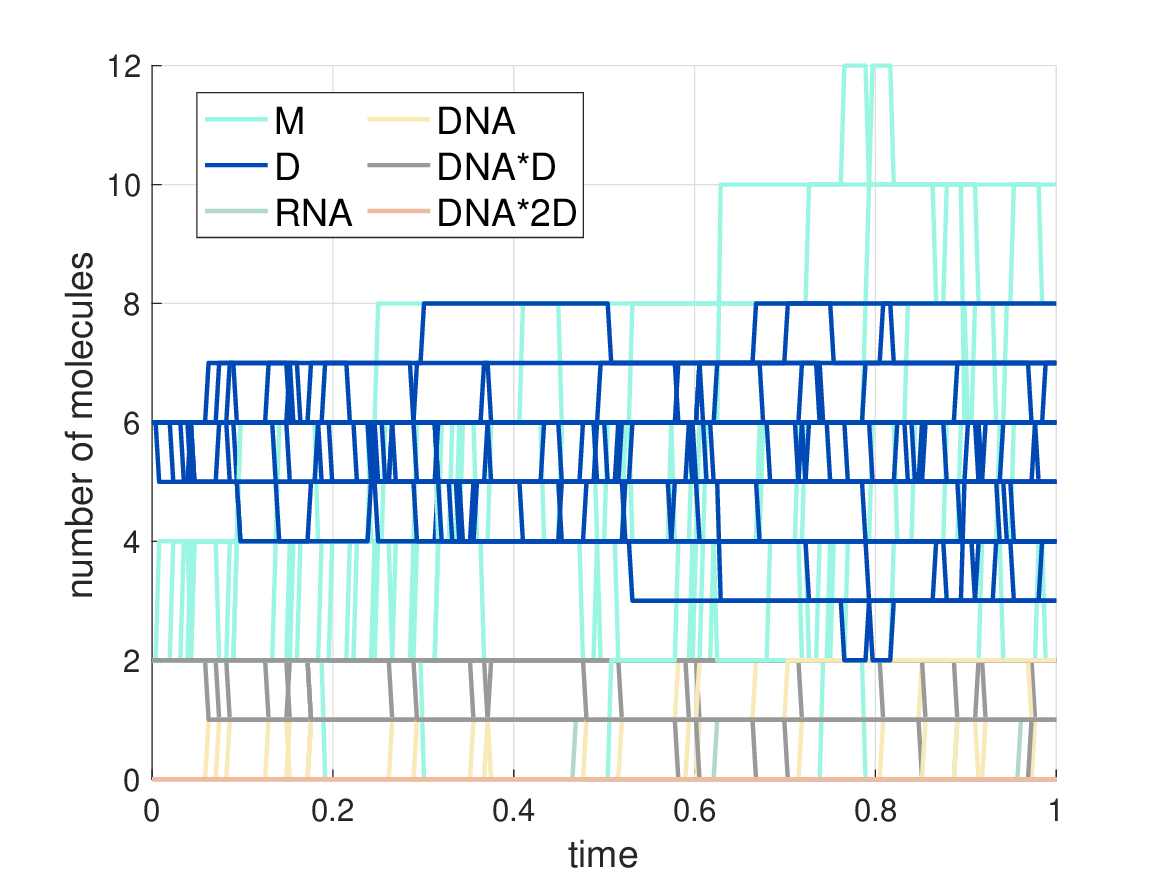}}
	\subfloat [Example \ref{exp:Goutsias}: 30 TL-MP paths] 
	{\includegraphics[scale=0.40]{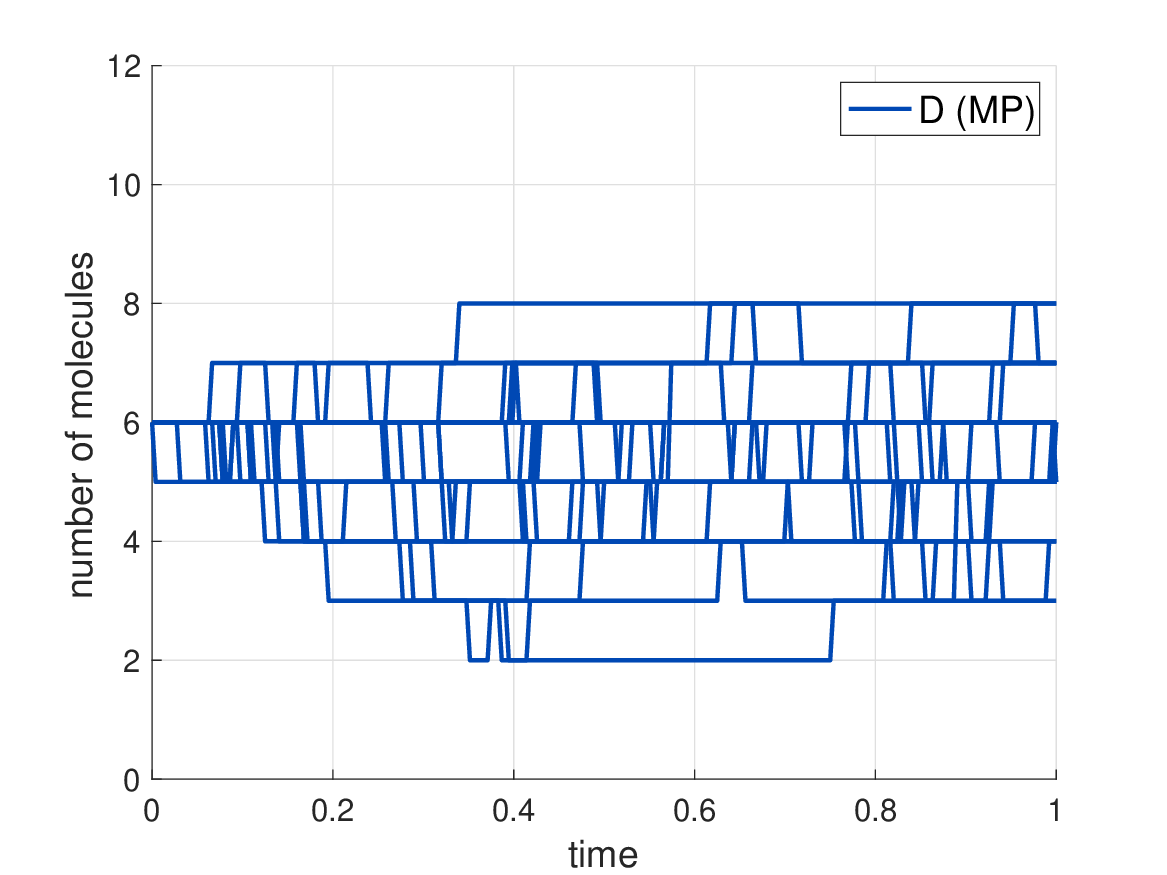}}
\end{figure}
\begin{figure}[h]
	\caption{Relative occurrences of states at final time $T$ in $M_{test}=10^4$ sample paths, comparing the TL estimate of $P(\mathbf{X}(t))\mid_{\{\mathbf{X}_0=\mathbf{x}_0\}}$ and the MP estimate of $\bar{\boldsymbol{S}}(T)\mid_{\{\mathbf{X}_0=\mathbf{x}_0\}}$. We set the step size to $\Delta t=2^{-8}$ for the sample paths. The MP is based on $M=10^4$ TL paths with a step size of $\Delta t=2^{-8}$ and uses the orthogonal basis of polynomials described in Remark~\ref{rem:orth_basis}, where $\Lambda = \{0,1,2\} \times \{0,1,2\}$. \label{fig:hist}}
	\subfloat [Example \ref{exp:mm}] 
	%{\includegraphics[scale=0.40]{figs/9_hist_24_104}}
	{\includegraphics[scale=0.40]{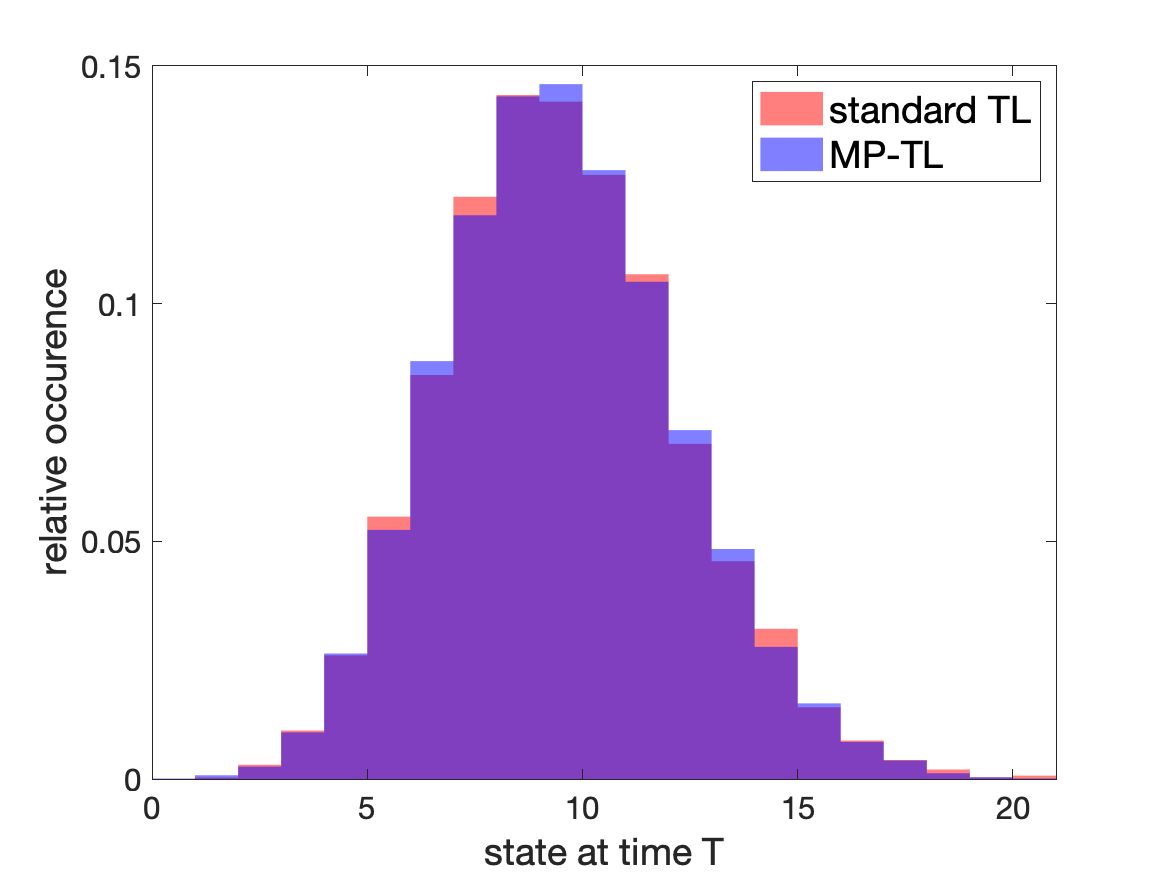}}
	\subfloat [Example \ref{exp:Goutsias}]
	{\includegraphics[scale=0.40]{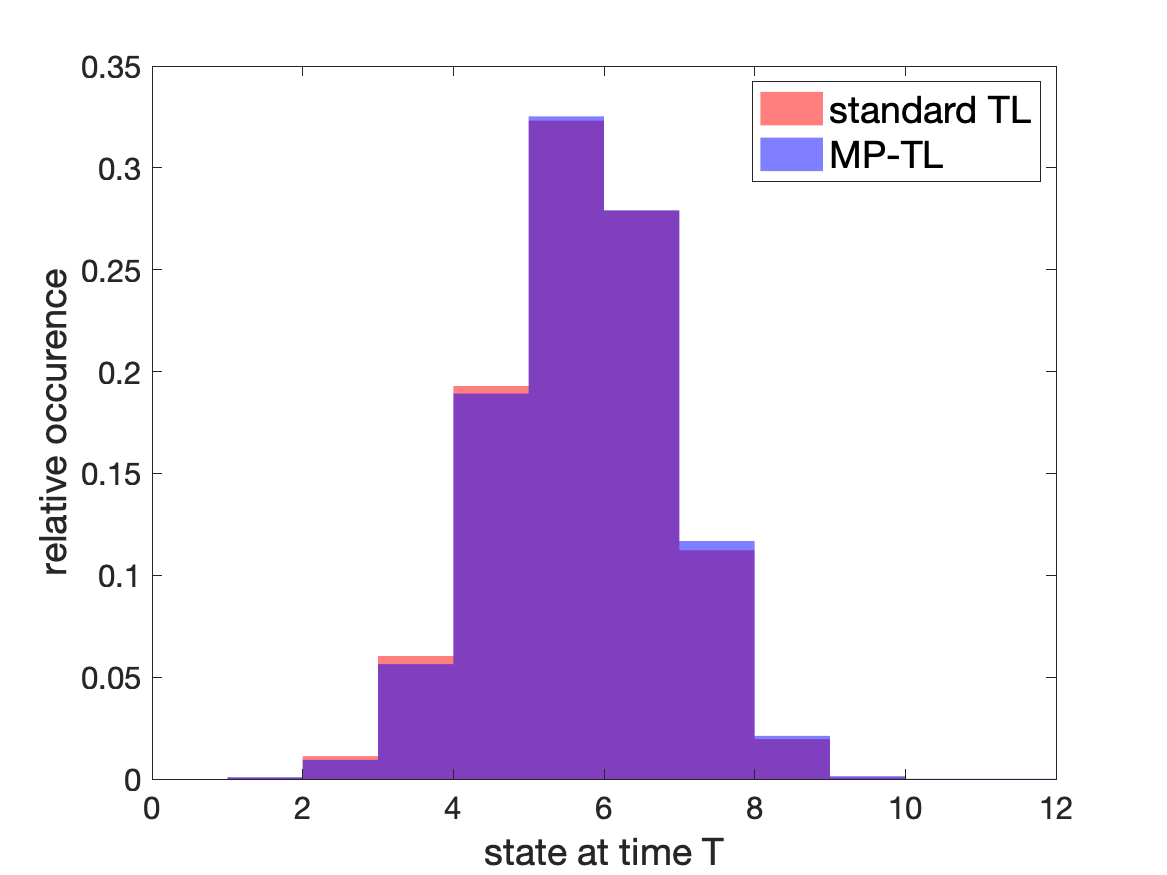}}
\end{figure}

\subsection{Makovian Projection-Importance Sampling Results}\label{sec:MPsim}
For the numerical experiments, we use a four-dimensional and a six-dimensional SRNs with the observable $g(\mathbf{x})=\mathbbm{1}_{\{x_i>\gamma\}}$, where $i$ and $\gamma$ are specified in Examples~\ref{exp:mm} and \ref{exp:Goutsias}. Figure~\ref{fig:hist} indicates that this observable leads to a rare event probability estimation for which an MC estimate is insufficient. We use the workflow in Figure~\ref{fig:scheme} with separate simulations for various $\Delta t$ values for the MP-IS simulations. The MP is based on $M=10^4$ TL sample paths each. The MP-IS-MC estimator, the sample variance, and the kurtosis estimate are based on $M_{fw}=10^6$ IS sample paths.

The relative error is more relevant than the absolute error for rare event probabilities. Therefore, we display the squared coefficient of variation~\cite{ben2023learning, ben2021efficient} in the simulations results, which is given by the following for a random variable $X$:
	\begin{align}
		Var_{rel}[X]=\frac{Var[X]}{\mathbb{E}[X]^2}.
	\end{align}
The kurtosis is a good indicator of the robustness of the variance estimator (see \cite{ben2020importance, ben2023learning} for the connection between the sample variance and kurtosis).

Figure~\ref{fig:mm} shows the simulation results for the four-dimensional Example \ref{exp:mm} for different step sizes $\Delta t$. The quantity of interest is a rare event probability with a magnitude of $10^{-5}$. For a step size of $\Delta t=2^{-10}$, the proposed MP-IS approach reduces the squared coefficient of variation by a factor of $10^6$ compared to the standard MC-TL approach. The third plot in Figure~\ref{fig:mm} indicates that the kurtosis of the proposed MP-IS approach is below the kurtosis for standard TL for all observed step sizes $\Delta t$, confirming that the proposed approach results in a robust variance estimator. In Figure~\ref{fig:mm} (d), we show the required number of sample paths to reach a prescribed relative tolerance (see \eqref{eq:relM} and \eqref{eq:reldt}). We compare the required number of sample paths for standard TL with the sample paths in the forward run of the IS-TL estimator, $M_{fw}$, and the total number of paths $M+M_{fw}$ including $M=10^4$ TL paths to derive the MP (see Section \ref{sec:comp_cost_MP}). We observe that for small tolerances, the additional paths to derive the MP become negligible compared to the cost of the forward IS-TL run. Further, we reduce the number of paths significantly compared to the standard TL approach. For small tolerances, the reduction of sample paths is up to a factor of approximately $10^6$ compared to the standard TL approach. 

The second application of the proposed IS approach is the six-dimensional Example \ref{exp:Goutsias}. Figure~\ref{fig:Goutsias} shows that this rare event probability has a magnitude of $10^{-3}$. We observe that, for $\Delta t \leq 2^{-3}$, the squared coefficient of variation of the proposed MP-IS approach is reduced compared to the standard TL-MC approach. For a step size of $\Delta t=2^{-10}$, this is a variance reduction of a factor of approximately $500$. Note that, this example achieves less variance reduction than Example \ref{exp:mm} due to a less rare quantity of interest. For most step sizes $\Delta t$, the kurtosis of the proposed IS approaches is moderately increased compared to the standard TL estimator, with decreasing kurtosis for smaller $\Delta t$. This outcome indicates a potentially insatiable variance estimator for coarse time steps of $\Delta t > 2^{-7}$. For finer time steps, we expect a robust variance estimator. For small tolerances the total number of paths is reduced by a factor of approximately $500$.

\begin{figure}[h!]
	\caption{Example~\ref{exp:mm} \label{fig:mm} with separate simulations for different step sizes $\Delta t$ for the proposed IS methods and the standard TL-MC estimator: (a) sample mean, (b) squared coefficient of variation, (c) kurtosis, and (d) number of required sample paths to reach a prescribed relative tolerance $\text{TOL}_{rel}$. The used MP is based on $M=10^4$ TL sample paths each, and the estimators in (a), (b) and (c) are derived based on $M_{fw}=10^6$ sample paths. For (d), we use \eqref{eq:reldt} and \eqref{eq:relM} and the dashed line represents the total number of sample paths including the derivation of the MP (i.e. $M+M^*_{rel}(\text{TOL}_{rel})$).}
	\subfloat[]{
		\includegraphics[scale=0.40]{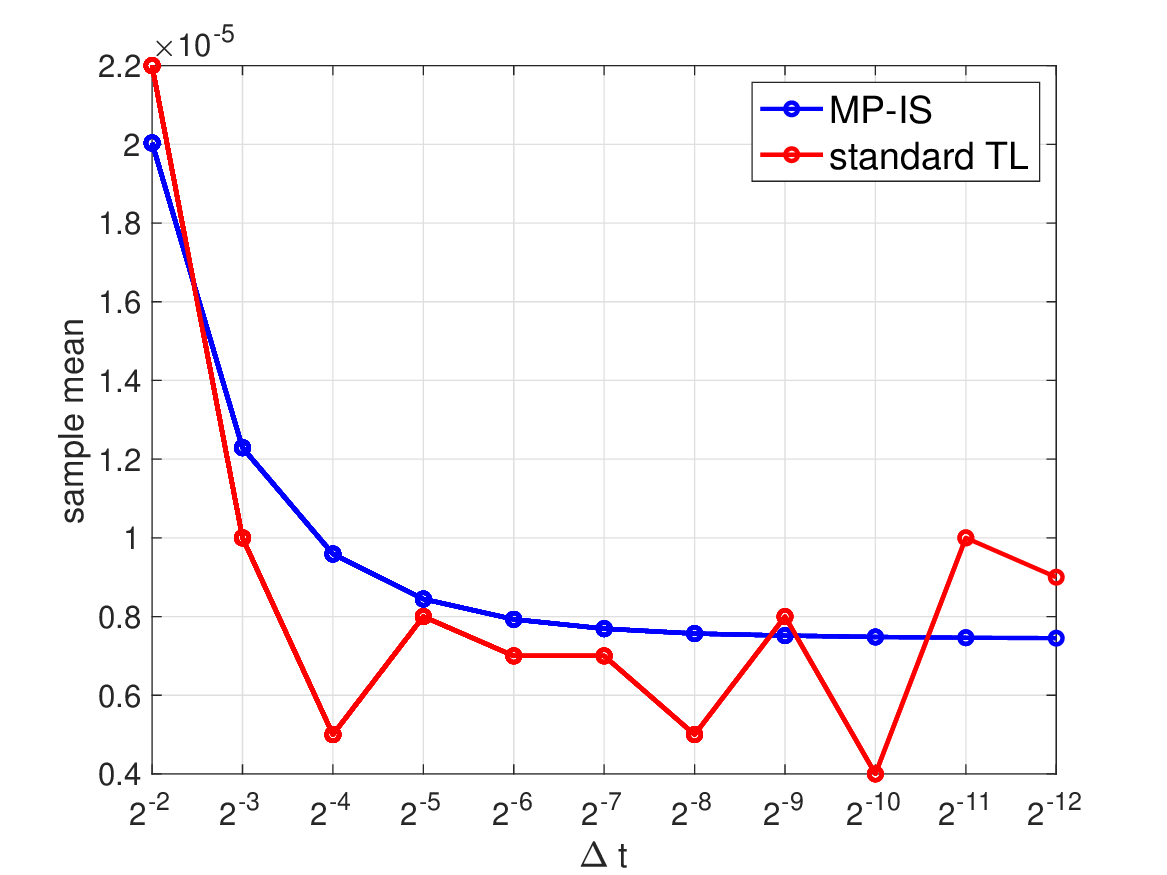}
	}
	\subfloat[]{
		\includegraphics[scale=0.40]{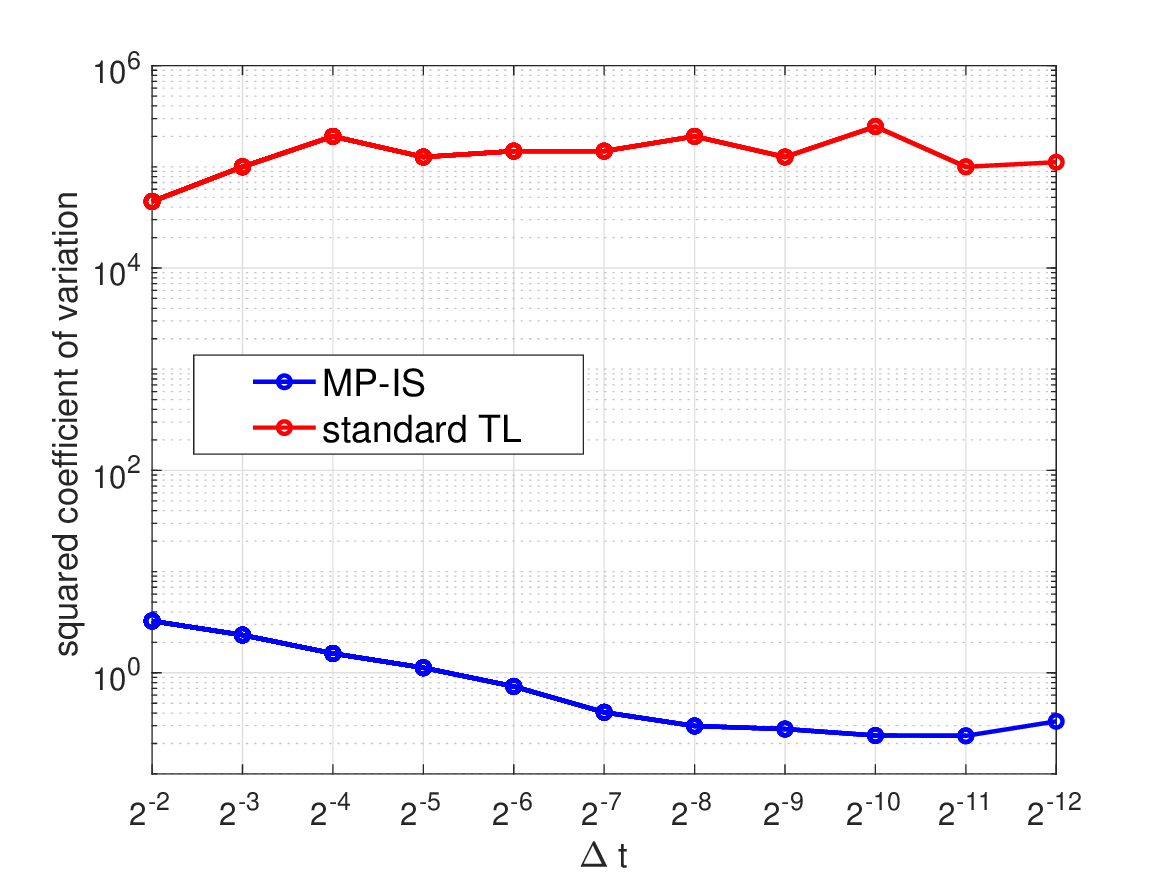}}\\
	\subfloat[]{	\includegraphics[scale=0.40]{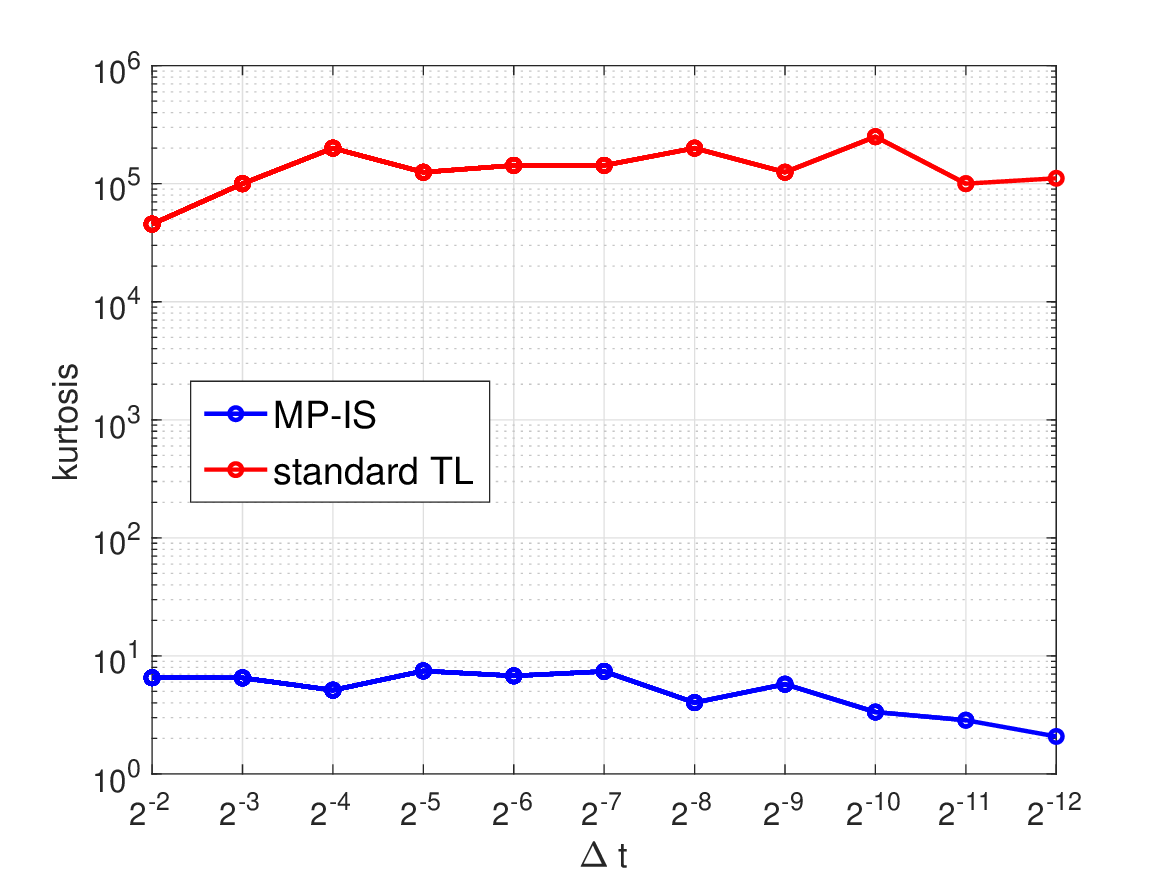}}		
		\subfloat[]{	\includegraphics[scale=0.40]{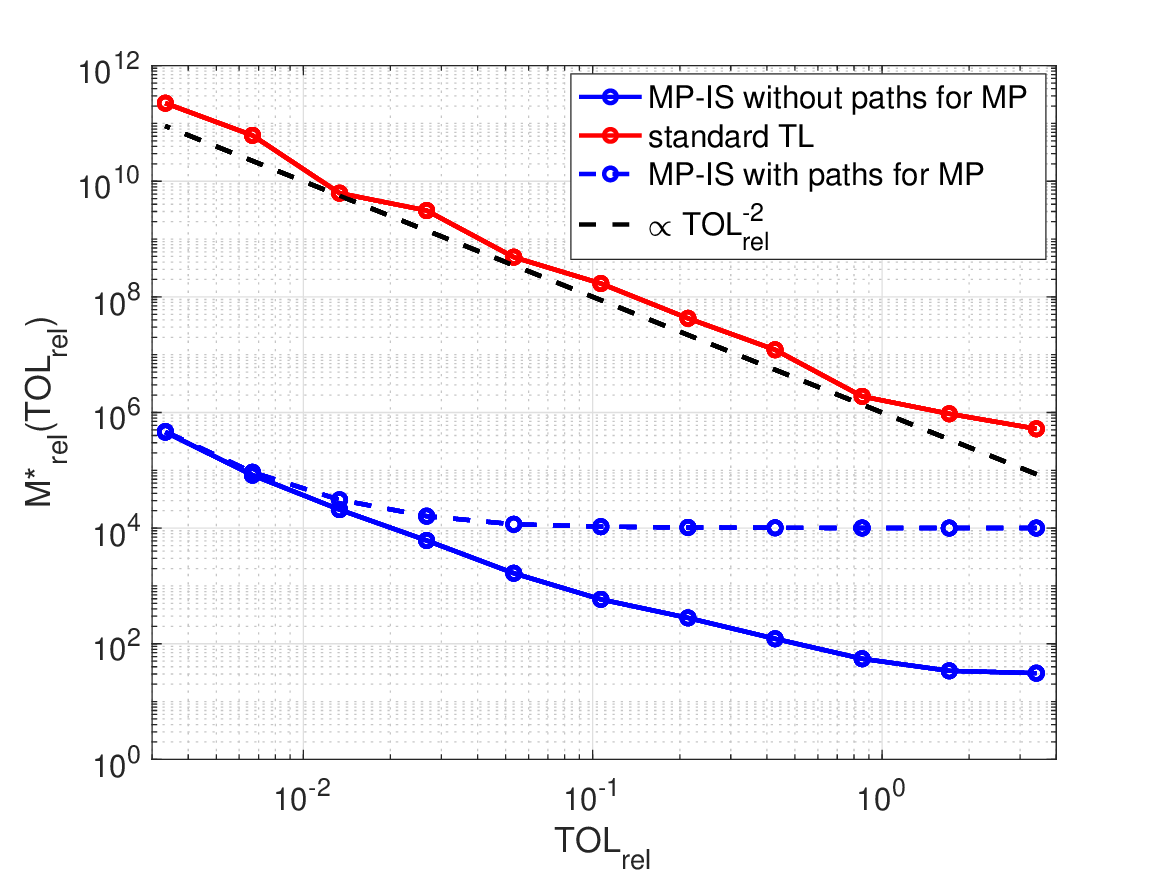}}
\end{figure}
\begin{figure}[h]
	\caption{Example \ref{exp:Goutsias} \label{fig:Goutsia} with separate simulations for different step sizes $\Delta t$ for the proposed IS methods and the standard TL-MC estimator: (a) sample mean, (b) squared coefficient of variation, (c) kurtosis, and (d) number of required sample paths to reach a prescribed relative tolerance $\text{TOL}_{rel}$. The used MP is based on $M=10^4$ TL sample paths each, and the estimators in (a), (b) and (c) are derived based on $M_{fw}=10^6$ sample paths. \label{fig:Goutsias}}
	\subfloat[]{\includegraphics[scale=0.40]{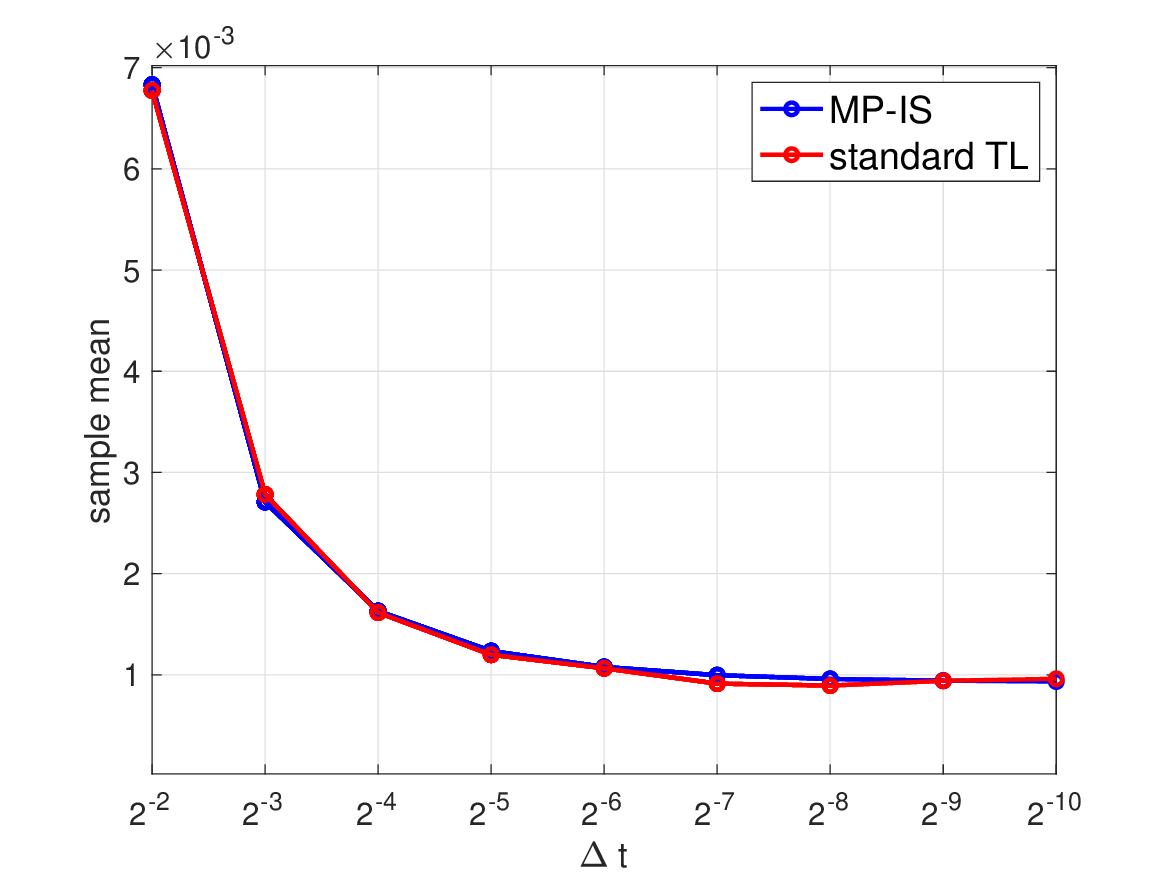}}
	\subfloat[]{	\includegraphics[scale=0.40]{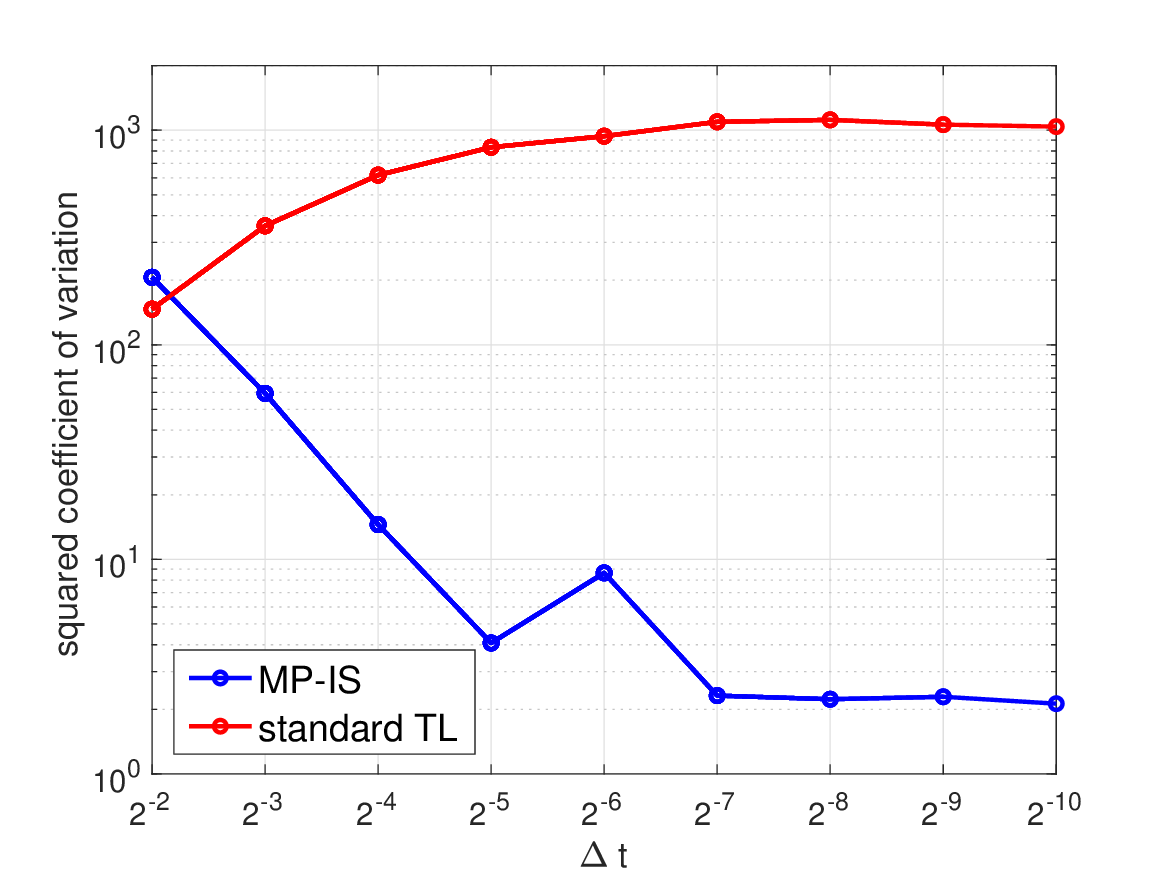}}\\
	\subfloat[]{	\includegraphics[scale=0.40]{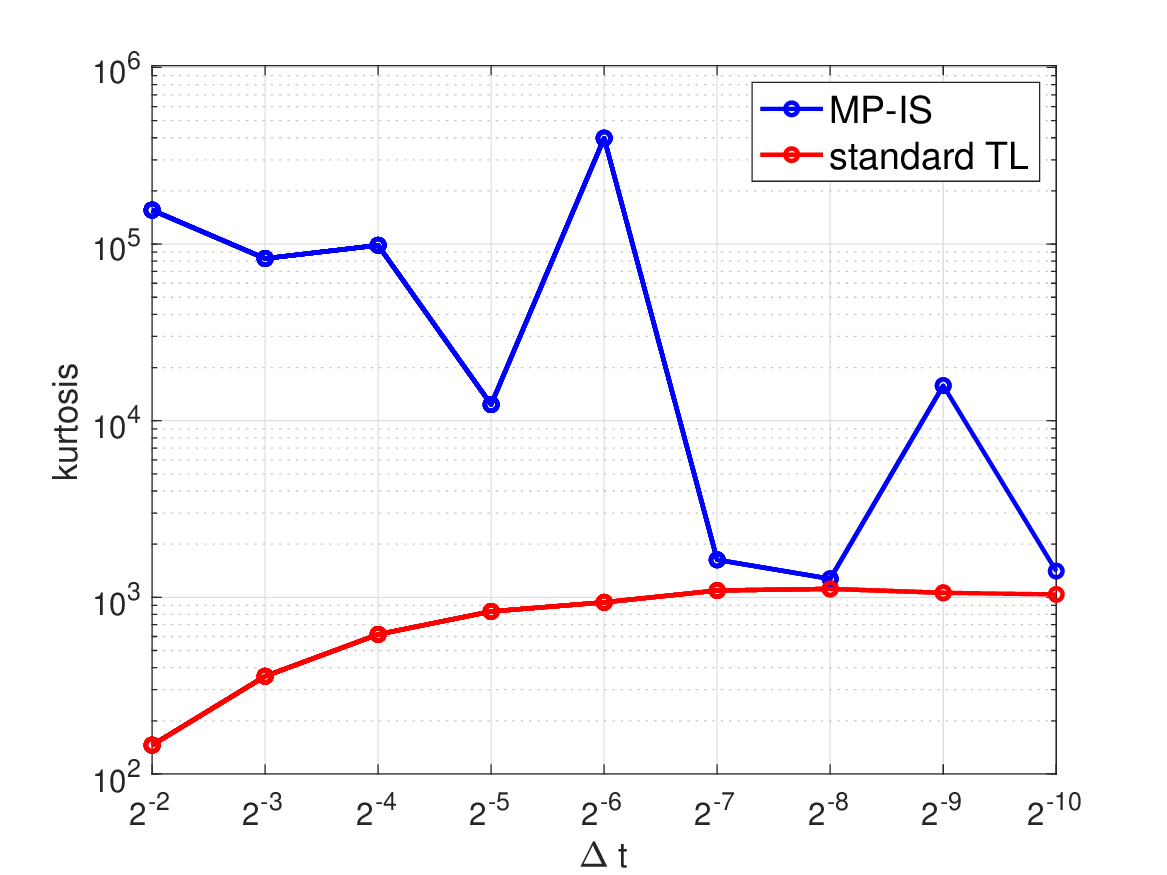}}
	\subfloat[]{	\includegraphics[scale=0.40]{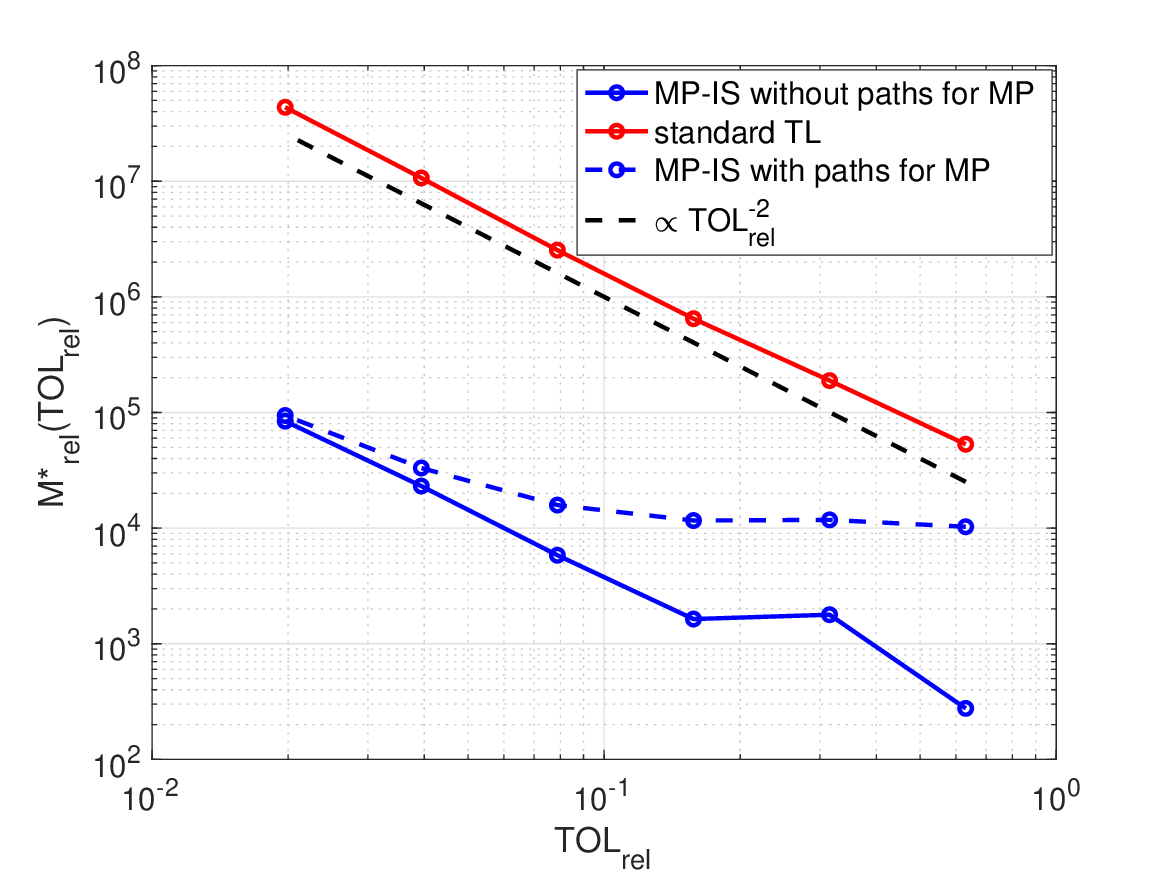}}
\end{figure}
\section{Conclusion}\label{sec:conclusion}
In conclusion, this work presented an efficient IS scheme for estimating rare event probabilities for SRNs. We utilized a class of parameterized IS measure changes originally introduced in \cite{ben2023learning}, for which near-optimal IS controls can be derived through a SOC formulation. We showed that the value function associated with this formulation can be expressed as a solution of a set of coupled ODEs, the HJB equations. One challenge encountered in solving the HJB equations is the curse of dimensionality, arising from the high-dimensional SRN. To address this issue, we introduced a dimension reduction approach for the setting of SRNs, namely MP. Then, we used a discrete $L^2$ regression to approximate the propensity and the stoichiometric vector of the MP-SRN. We demonstrated how the MP-SRN can be used for solving a significantly lower-dimensional HJB system, and how the resulting parameters are then mapped back to the full-dimensional SRNs to derive near-optimal IS controls. Our numerical simulations showed substantial variance reduction for the MP-IS-MC estimator compared to the standard MC-TL estimator for rare event probability estimations.

%%%%%%%%%%%%%%%%%%%%%%%%%%%%%%%%%%%%%%%%%%
\vspace{1cm}
\textbf{Acknowledgments}
This publication is based upon work supported by the King Abdullah University of Science and Technology (KAUST) Office of Sponsored Research (OSR) under Award No. OSR-2019-CRG8-4033. This work was performed as part of the Helmholtz School for Data Science in Life, Earth and Energy (HDS-LEE) and received funding from the Helmholtz Association of German Research Centres and the Alexander von Humboldt Foundation. For the purpose of open access, the authors have applied a Creative Commons Attribution (CC BY) licence to any Author Accepted Manuscript version arising from this submission.
%%%%%%%%%%%%%%%%%%%%%%%%%%%%%%%%%%%%%%%%%%

%References
%%%%%%%%%%%%%%%%%%%%%%%%%%%%%%%%%%%%%%%%%%
\bibliographystyle{plain}
\bibliography{MP_IS_manuscript} 

\appendix
\section{Proof for Corollary~\ref{prop:HJB}}\label{Appendix:HJB}
\begin{proof}
	For $\mathbf{x} \in \mathbb{N}^d$, we define $\tilde{u}(\cdot, \mathbf{x};\Delta t)$ as the  continuous smooth extension of $u_{\Delta t}(\cdot, \mathbf{x})$  (defined in \eqref{eq:exact_optival})  on $[0,T]$. Consequently, we denote the continuous-time IS controls by $\boldsymbol{\delta}(\cdot,\mathbf{x}): [0,T]\rightarrow \mathcal{A}_{\mathbf{x}}$ for $\mathbf{x}\in\mathbb{N}^d$.
	Then, the Taylor expansion of $\tilde{u}(t+\Delta t, \mathbf{x};\Delta t)$ in $t$ results in the following:
	\begin{align}\label{eq:Tayloru}
		\tilde{u}(t+\Delta t, \mathbf{x};\Delta t)=\tilde{u}(t, \mathbf{x};\Delta t)+\Delta t \partial_t \tilde{u}(t, \mathbf{x};\Delta t)+\mathcal{O}\left(\Delta t^2\right), \mathbf{x}\in \mathbb{N}^d.
	\end{align}
	
	By the definition of the value function \ref{def:optival2}, the final condition is given by
	\begin{align*}
		\tilde{u}(T, \mathbf{x};\Delta t)=g^2(\mathbf{x}), \mathbf{x} \in \mathbb{N}^d.
	\end{align*}
	
	For $t=T-\Delta t, \ldots, 0$, we apply to \eqref{eq:exact_optival} from Theorem \ref{theo:exact_optival} a Taylor expansion around $\Delta t=0$ to the exponential term and \eqref{eq:Tayloru}  to $\tilde{u}(t+\Delta t, \mathbf{x};\Delta t)$:
	\begin{small}
		\begin{align}\label{eq:HJBproof}
			\tilde{u}(t, \mathbf{x};\Delta t)&=
			\inf_{\boldsymbol{\delta}(t,\mathbf{x})\in\mathcal{A}_\mathbf{x}}\exp\left(\left(-2\sum_{j=1}^J a_j(\mathbf{x})+\sum_{j=1}^J\delta_j(t,\mathbf{x})\right)\Delta t\right) \nonumber\\
			&\quad \quad \quad \quad \times \sum_{\mathbf{p} \in \mathbb{N}^J}\left(\prod_{j=1}^{J} \frac{(\Delta t \cdot \delta_j(t,\mathbf{x}))^{p_j}}{p_j!} \left(\frac{a_j(\mathbf{x})}{\delta_j(t,\mathbf{x})}\right)^{2p_j} \right)\cdot \tilde{u}(t+\Delta t,\max(\mathbf{0},\mathbf{x}+ \boldsymbol{\nu}\mathbf{p});\Delta t)\nonumber\\
			&=
			\inf_{\boldsymbol{\delta}(t,\mathbf{x}) \in \mathcal{A}_\mathbf{x}} \left(1+ \left(-2 \sum_{j=1}^J a_j(\mathbf{x})+\sum_{j=1}^J \delta_ j(t,\mathbf{x})\right)\Delta t +\mathcal{O}(\Delta t^2)\right) \nonumber\\
			&\quad \quad \quad \quad\times \left[\sum_{\mathbf{p}\in\mathbb{N}^J}\left(\prod_{j=1}^{J} \frac{(\Delta t \cdot \delta_ j(t,\mathbf{x}))^{p_j}}{p_j!} \left(\frac{a_j(\mathbf{x})}{ \delta_ j(t,\mathbf{x})}\right)^{2p_j} \right)\right. \nonumber\\
			&\quad \quad \quad \quad\left.\vphantom{\prod_{j=1}^{J}} \cdot\left( \tilde{u}\left(t, \max( \mathbf{0},\mathbf{x}+\boldsymbol{\nu}\mathbf{p});\Delta t\right)+\Delta t \partial_t 	\tilde{u}\left(t, \max( \mathbf{0},\mathbf{x}+\boldsymbol{\nu}\mathbf{p});\Delta t \right)+\mathcal{O}(\Delta t^2)\right)\right] \nonumber\\
			\overset{(*)}{\Longrightarrow}- \partial_t &\tilde{u}(t, \mathbf{x};\Delta t) =\inf _{\boldsymbol{\delta}(t,\mathbf{x}) \in \mathcal{A}_\mathbf{x}}\left(-2 \sum_{j=1}^J a_j(\mathbf{x})+\sum_{j=1}^J  \delta_ j(t,\mathbf{x})\right)	\tilde{u}(t, \mathbf{x};\Delta t)+\mathcal{O}(\Delta t) \nonumber\\
			&\quad \quad \quad \quad\quad \quad+(1+\mathcal{O}(\Delta t)) \left[\sum_{\mathbf{p} \neq \mathbf{0}} \Delta t^{|\mathbf{p}|-1}\left(\prod_{j=1}^J \frac{a_j(\mathbf{x})^{2 p_j}}{p_{j} ! \cdot\left( \delta_ j(t,\mathbf{x})\right)^{p_j}}\right)\right.\nonumber\\
			&\quad \quad \quad \quad \quad\quad \left.\vphantom{\prod_{j=1}^{J}}\cdot \left(	\tilde{u}\left(t, \max( \mathbf{0},\mathbf{x}+\boldsymbol{\nu}\mathbf{p});\Delta t\right) +\mathcal{O}(\Delta t) \right)\right],
		\end{align}
	\end{small}
	where $|\mathbf{p}|:=\sum_{i=1}^Jp_j$, and $ \delta_ j(t,\mathbf{x}):=\left(\boldsymbol{\delta}(t,\mathbf{x})\right)_j$. In (*), we split the sum, rearrange the terms, divide by $\Delta t$, and collect the terms of $\mathcal{O}(\Delta t)$.
	
	The limit for $\Delta t \rightarrow 0$ in \eqref{eq:HJBproof} is denoted by $\tilde{u}(t,\mathbf{x})$, leading to \eqref{eq:HJB} for $0<t<T$ and $\mathbf{x}\in\mathbb{N}^d$.
\end{proof}

\section{Proof for Theorem~\ref{theo:MP}}\label{Appendix:MP}
\begin{proof}
	We let $f:\mathbb{R}^{\bar{d}}\rightarrow \mathbb{R}$ be an arbitrary bounded continuous function and $\bar{\boldsymbol{S}}$ be defined in \eqref{eq:proj}. We consider the following weak approximation error: 
	\begin{align}\label{eq:weakerror}%this is (-1)* the error from the SDE slides
		\varepsilon_T:=\mathbb{E}\left[f(P(\mathbf{X}(T))) \middle|\mathbf{X}(0)=\mathbf{x}_0 \right]-\mathbb{E}\left[f(\bar{\boldsymbol{S}}(T)) \middle|\bar{\boldsymbol{S}}(0)=P(\mathbf{x}_0)\right].
	\end{align}
	For $t \in [0,T]$, we define the cost to go function as 
	\begin{align*}
		\bar{v}(t,\boldsymbol{s}):=\mathbb{E}\left[ f(\bar{\boldsymbol{S}}(T))\middle|\bar{\boldsymbol{S}}(t)=\boldsymbol{s} \right].
	\end{align*}
	Then, we can represent the weak error in \eqref{eq:weakerror} as follows:
	\begin{align}\label{eq:weakerror2}
		\varepsilon_T=\mathbb{E}\left[\bar{v}(T,P(\mathbf{X}(T)))\middle|\mathbf{X}(0)=\mathbf{x}_0\right]-\bar{v}(0,P(\mathbf{x}_0)).
	\end{align}
	Using Dynkin's formula \cite{doi:10.1137/1.9780898718638} and \eqref{eq:exact_process}, we can express the first term in \eqref{eq:weakerror} as follows:
	\begin{align*}
		\mathbb{E}&\left[\bar{v}(T,P(\mathbf{X}(T)))\middle|\mathbf{X}(0)=\mathbf{x}_0\right]\\
		&=\bar{v}(0,P(\mathbf{x}_0))+\int_0^T \mathbb{E}\left[\vphantom{\sum_{j=1}^J} \partial_t\bar{v}(\tau,P(\mathbf{X}(\tau)))\right.\\
		&\left.+\sum_{j=1}^J a_j(\mathbf{X}(\tau))\left( \bar{v}(\tau,P(\mathbf{X}(\tau)+\boldsymbol{\nu}_j))-\bar{v}(\tau,P(\mathbf{X}(\tau)))\right) \middle|\mathbf{X}(0)=\mathbf{x}_0\right]d\tau.
	\end{align*}
	The Kolmogorov backward equations  \cite{doi:10.1137/1.9780898718638} of $\bar{\boldsymbol{S}}$ are given as
	\begin{align*}
		\partial_{\tau}\bar{v}(\tau,\boldsymbol{s})=-\sum_{j=1}^J \bar{a}_j(\tau,\boldsymbol{s})\left( \bar{v}(\tau,\boldsymbol{s}+\bar{\boldsymbol{\nu}}_j)-\bar{v}(\tau,\boldsymbol{s})\right), \mathbf{s} \in \mathbb{N}^{\bar{d}},
	\end{align*}
	implying that the weak error simplifies to
	\begin{small}
		\begin{align}\label{eq:epst}
			\varepsilon_T
			&=\sum_{j=1}^J \int_0^T \mathbb{E}\left[a_j(\mathbf{X}(\tau))\bar{v}(\tau,P(\mathbf{X}(\tau)+\boldsymbol{\nu}_j))-\bar{a}_j(\tau,P(\mathbf{X}(\tau))) \bar{v}(\tau,P(\mathbf{X}(\tau))+\bar{\boldsymbol{\nu}}_j)\middle|\mathbf{X}(0)=\mathbf{x}_0\right]\nonumber\\
			&-\mathbb{E}\left[\left(a_j(\mathbf{X}(\tau))-\bar{a}_j(\tau,P(\mathbf{X}(\tau)))\right)\bar{v}(\tau,P(\mathbf{X}(\tau))) \middle|\mathbf{X}(0)=\mathbf{x}_0\right]d\tau.
		\end{align}
	\end{small}
	Next, we choose $\bar{a}_j$ and $\bar{\boldsymbol{\nu}}_j$ for $j=1,\dots,J$ such that $\varepsilon_T=0$ for any function $f$. We consider the second term in \eqref{eq:epst} and use the tower property to obtain
	\begin{align}\label{eq:term2}
		&\mathbb{E}\left[\left(a_j(\mathbf{X}(\tau))-\bar{a}_j(\tau,P(\mathbf{X}(\tau)))\right)\bar{v}(\tau,P(\mathbf{X}(\tau))) \middle| \mathbf{X}(0)=\mathbf{x}_0\right]\nonumber\\
		&=\mathbb{E}\left[ \mathbb{E}\left[ \left(a_j(\mathbf{X}(\tau))-\bar{a}_j(\tau,P(\mathbf{X}(\tau)))\right)\bar{v}(\tau,P(\mathbf{X}(\tau)))\middle| P(\mathbf{X}(\tau)), \mathbf{X}(0)=\mathbf{x}_0\right]\middle| \mathbf{X}(0)=\mathbf{x}_0\right]\nonumber\\
		&=\mathbb{E}\left[\left( \mathbb{E}\left[{a}_j(\mathbf{X}(\tau))\middle| P(\mathbf{X}(\tau)), \mathbf{X}(0)=\mathbf{x}_0\right]-\bar{a}_j(\tau,P(\mathbf{X}(\tau))) \right)\bar{v}(\tau,P(\mathbf{X}(\tau)))\middle|\mathbf{X}(0)=\mathbf{x}_0\right].
	\end{align}
	To ensure that \eqref{eq:term2}$=0$ for any function $f$, we obtain the following:
	\begin{align}\label{eq:proofabar}
		\bar{a}_j(\tau,P(\mathbf{X}(\tau)))= \mathbb{E}\left[{a}_j(\mathbf{X}(\tau))\middle| P(\mathbf{X}(\tau)), \mathbf{X}(0)=\mathbf{x}_0\right], j=1,\dots,J.
	\end{align}
	Applying \eqref{eq:proofabar} and the tower property for the first term, we derive
	\begin{align}\label{eq:term1}
		&	\mathbb{E}\left[a_j(\mathbf{X}(\tau))\bar{v}(\tau,P(\mathbf{X}(\tau)+\boldsymbol{\nu}_j))-\bar{a}_j(\tau,P(\mathbf{X}(\tau))) \bar{v}(\tau,P(\mathbf{X}(\tau))+\bar{\boldsymbol{\nu}}_j)\middle|\mathbf{X}(0)=\mathbf{x}_0\right]\nonumber\\
		&=\mathbb{E}\left[ \mathbb{E}\left[a_j(\mathbf{X}(\tau))\bar{v}(\tau,P(\mathbf{X}(\tau)+\boldsymbol{\nu}_j))\right.\right.\nonumber\\
		&\left.\left.-\bar{a}_j(\tau,P(\mathbf{X}(\tau))) \bar{v}(\tau,P(\mathbf{X}(\tau))+\bar{\boldsymbol{\nu}}_j)\middle| P(\mathbf{X}(\tau)), \mathbf{X}(0)=\mathbf{x}_0\right]\middle|\mathbf{X}(0)=\mathbf{x}_0\right]\nonumber\\
		&=\mathbb{E}\left[ \mathbb{E}\left[a_j(\mathbf{X}(\tau))\middle|P(\mathbf{X}(\tau)),  \mathbf{X}(0)=\mathbf{x}_0\right]\bar{v}(\tau,P(\mathbf{X}(\tau))+P(\boldsymbol{\nu}_j)))\right.\nonumber\\
		&\left.-\bar{a}_j(\tau,P(\mathbf{X}(\tau))) \bar{v}(\tau,P(\mathbf{X}(\tau))+\bar{\boldsymbol{\nu}}_j)\middle|\mathbf{X}(0)=\mathbf{x}_0\right]\nonumber\\
		&=\mathbb{E}\left[ \vphantom{\underbrace{\bar{v}}_{ \bar{\boldsymbol{\nu}}_j}} \mathbb{E}\left[a_j(\mathbf{X}(\tau))\middle|P(\mathbf{X}(\tau)),  \mathbf{X}(0)=\mathbf{x}_0\right] \right.\nonumber\\
		&~~~~~~~~~~\cdot  \left.\left(\bar{v}(\tau,P(\mathbf{X}(\tau))+P(\boldsymbol{\nu}_j))-\bar{v}(\tau,P(\mathbf{X}(\tau))+\bar{\boldsymbol{\nu}}_j) \right)\middle| \mathbf{X}(0)=\mathbf{x}_0\right].
	\end{align}
	Moreover, Equation \eqref{eq:term1}  becomes zero for any function $f$ using 
	\begin{align}
		\bar{\boldsymbol{\nu}}_j=P(\boldsymbol{\nu}_j), j=1,\dots,J.
	\end{align}
	
	With this choice for $\bar{a}_j$ and $\bar{\boldsymbol{\nu}}_j$, we derive $\varepsilon_T=0$. The derivation holds for arbitrary bounded and smooth functions $f$, for all fixed times $T$; thus, the process $\boldsymbol{S}(t)=P(\mathbf{X}(t))$ has the same conditional distribution as $\bar{\boldsymbol{S}}(t)$ conditioned on the initial value $\mathbf{X}(0)=\mathbf{x}_0$.\\
\end{proof}

\section{Markovian Projection Cost Derivation}\label{appendix:cost}

We present details on the computational cost of MP, as provided in \eqref{eq:MPcost}:

\begin{itemize}
	\item The number of operations to generate one TL paths is given by 
	\begin{align}\label{eq:appendixWtl}
		W_{TL}(\Delta t)=\frac{T}{\Delta t} \cdot (C_{prop}+J\cdot C_{Poi}+d(J+2)),
	\end{align}
	where $C_{prop}$ is the cost of one evaluation of the propensity function \eqref{eq:prop_dynamics}. The dominant cost in \eqref{eq:appendixWtl} is $C_{Poi}$ (the cost of generating a Poisson random variable).
	\item The number of operations for the Gram--Schmidt algorithm, as described in Remark~\ref{rem:orth_basis}, is given by 
	\begin{align}\label{eq:appendixWgs}
		W_{G-S}(\#\Lambda,\Delta t,M)=\#\Lambda\cdot(C_{inner}+\#\Lambda+1)+\frac{(\#\Lambda-1)\#\Lambda}{2}(2\#\Lambda+C_{inner}),
	\end{align}
	where $C_{inner}$ is the cost of the evaluation of the empirical inner product \eqref{eq:skalar} given by
	\begin{align}\label{eq:appendixcinner}
		C_{inner}=\frac{T}{\Delta t} \cdot M (2+2C_{pol})+3 = \mathcal{O}(\frac{T}{\Delta t} \cdot M \cdot \#\Lambda).
	\end{align}
	The cost $C_{pol}$ in \eqref{eq:appendixcinner} is the computational cost for one evaluation of a polynomial in the space $<\phi_p>_{p\in\Lambda}$, which is $\mathcal{O}(\#\Lambda)$.

	In the simulations, we apply the setting $\#\Lambda \ll \frac{T}{\Delta t} \cdot M$ (see Section~\ref{sec:MPsim}, using the parameter $\#\Lambda=9$, $M=10^4$, $\frac{T}{\Delta t}=2^4$). Therefore, the dominant cost in \eqref{eq:appendixWgs} is $\mathcal{O}(M \cdot \frac{T}{\Delta t}\cdot  \left(\#\Lambda\right)^3)$.
	\item The cost $W_{L^2}(\#\Lambda,\Delta t,M)$ is split into two: the cost to (i) derive and (ii) solve the normal equation \eqref{eq:normal}. The number of operations to derive the design matrix $D$ is $M\cdot\frac{T}{\Delta t}\cdot\# \Lambda \cdot C_{pol}$, and the cost to derive one right-hand side $(\Psi^{(j)})_{j\in\mathcal{J}_{MP}}$ is $M\cdot\frac{T}{\Delta t}\cdot C_{prop}$. In \eqref{eq:normal}, the cost for the matrix product $D^\top D$ is $\mathcal{O}(\#\Lambda^2\cdot M\cdot \frac{T}{\Delta t})$, and the cost for $\#\mathcal{J}_{MP}$ matrix-vector products is $\mathcal{O}(\#\mathcal{J}_{MP}\cdot\#\Lambda\cdot  M\cdot \frac{T}{\Delta t})$. Finally, solving \eqref{eq:normal} costs $\mathcal{O}(\#\mathcal{J}_{MP} \cdot \#\Lambda^3)$, which is a nondominant term under the given setting, $\#\Lambda \ll \frac{T}{\Delta t} \cdot M$.
 %\item and is approximately given by 
	%\begin{align*}
	%{	\color{green}W_{HJB}(\#\Lambda)\approx(T/ \Delta t)\cdot \prod_{i=1}^{\bar{d}}\bar{S}_i\cdot J,}
	%	\end{align*}
%where the infinite state space $\mathbb{N}^{\bar{d}}$ is truncated to $\times_{i=1}^{\bar{d}}[0, \bar{S}_i ]$ for sufficiently large upper bounds $\bar{S}_1,\dots,\bar{S}_{\bar{d}}$. 
\end{itemize}

\end{document}